\documentclass{amsart}
\usepackage[dvips]{epsfig}
\usepackage{graphics}
\usepackage{latexsym}
\usepackage{amsmath}
\usepackage{amsthm}
\usepackage{amssymb}

\newtheorem{thm}{Theorem}
\newtheorem{lem}[thm]{Lemma}

\newtheorem{defi}[thm]{Definition}

\newtheorem{prop}[thm]{Proposition}

\newenvironment{preuve}{\vip \noindent {\it Proof}}{\hfill$\square$\vip}

\newcommand{\vip}{\vskip.2cm}

\newcommand{\rr}{{\mathbb{R}}}

\newcommand{\tN}{{\tilde N}}
\newcommand{\tQ}{{\tilde Q}}

\newcommand{\tY}{{\tilde Y}}
\newcommand{\tX}{{\tilde X}}
\newcommand{\tI}{{\tilde I}}
\newcommand{\tK}{{\tilde K}}
\newcommand{\tL}{{\tilde L}}
\newcommand{\tJ}{{\tilde J}}
\newcommand{\tx}{{\tilde x}}
\newcommand{\ty}{{\tilde y}}

\newcommand{\E}{\mathbb{E}}
\newcommand{\cF}{{\mathcal F}}
\newcommand{\cG}{{\mathcal G}}

\newcommand{\intot}{\int_0^t}
\newcommand{\intrs}{\int_{\rr_*}}

\newcommand{\sm}{{s-}}
\newcommand{\indiq}{{{\bf 1}}}
\newcommand{\ala}{ \\ \nonumber }
\newcommand{\sg}{{\rm sign}}

\begin{document}

\title[Stable-driven SDEs]
{On pathwise uniqueness for stochastic differential equations
driven by stable L\'evy processes}

\author{Nicolas Fournier}

\address{Nicolas Fournier, LAMA UMR 8050,
Facult\'e de Sciences et Technologies,
Universit\'e Paris Est, 61 avenue du G\'en\'eral de Gaulle, 94010 Cr\'eteil 
Cedex, France, 
email: {\tt nicolas.fournier@univ-paris12.fr}}

\begin{abstract}
We study a one-dimensional stochastic differential equation
driven by a
stable L\'evy process of order $\alpha$ with drift and diffusion coefficients
$b,\sigma$.
When $\alpha\in (1,2)$, we investigate pathwise uniqueness for this equation.
When $\alpha\in (0,1)$, we study another
stochastic differential equation, which is equivalent in law, but for which
pathwise uniqueness holds under much weaker conditions.
We obtain various results, depending on whether $\alpha\in (0,1)$
or $\alpha \in (1,2)$ and on whether the driving stable process is symmetric or not.
Our assumptions involve the regularity and monotonicity of
$b$ and $\sigma$.
\end{abstract}

\maketitle



\section{Introduction and results}

For $a_-,a_+$ in $[0,\infty)$ and $\alpha \in (0,2)\setminus\{1\}$,
we consider the measure  on $\rr_*$:
\begin{align}\label{levymeas}
\nu_{a_-,a_+}^\alpha(dz)= |z|^{-\alpha-1}
[a_- \indiq_{\{z<0\}} + a_+ \indiq_{\{z>0\}}]dz.
\end{align}
Let now  $N(dsdz)$ be a Poisson measure on $[0,\infty)\times \rr_*$ with
intensity measure $ds \; \nu_{a_-,a_+}^\alpha(dz)$. Setting 
\begin{align}\label{stable}
\left\{ \begin{array}{l}
Z_t =\intot \intrs z N(dsdz) \quad \hbox{if } \alpha\in (0,1), \\
Z_t = \intot \intrs z \tN(dsdz)  \quad \hbox{if } \alpha\in (1,2),
\end{array} \right.
\end{align}
the process  $(Z_t)_{t\geq 0}$ is a {\it stable process} of order
$\alpha$ with parameters $a_-,a_+$, or a $(\alpha,a_-,a_+)$-stable process
in short. It is said to be {\it symmetric} if $a_-=a_+$.
Here $\tN$ stands for the
compensated Poisson measure, see Jacod-Shiryaev \cite[Chapter II]{js}.
We refer to Bertoin \cite{ber} and Sato \cite{sa} 
for many details on stable processes.
We consider, for some measurable functions
$\sigma,b:\rr\mapsto \rr$, the S.D.E.
\begin{align}\label{sde}
X_t &= x + \intot \sigma(X_\sm)dZ_s  + \intot b(X_s)ds.
\end{align}
Our aim in this paper is to investigate pathwise uniqueness
for this equation. 
Let us recall briefly the known results on this topic.

\vip

$\bullet$ Pathwise uniqueness classically holds when 
$b,\sigma$ are both Lipschitz-continuous, see e.g. Ikeda-Watanabe 
\cite[Chapter 4]{iw},
Protter \cite[Chapter 5]{p}.

$\bullet$ When $\alpha\in (1,2)$, $a_+=a_-$ and $b=0$,
Komatsu \cite{k} has shown pathwise
uniqueness if $\sigma$ is H\"older-continuous with
index $1/\alpha$, see also Bass \cite{b1}.

$\bullet$ Bass-Burdzy-Chen \cite{bbc} have proved that the above
results are sharp: if $a_-=a_+$ and $b=0$, for any 
$\beta<\min(1,1/\alpha)$,
one can find a function $\sigma$, H\"older-continuous with index $\beta$,
bounded from above and from below, such that 
pathwise uniqueness fails for (\ref{sde}).

\vip

We refer to the review paper of Bass \cite{b2} for many 
more details on the subject and to Situ \cite{s} for a book on 
general S.D.E.s with jumps.

\subsection{Preliminaries}
When $\alpha \in (1,2)$, we will study the S.D.E. (\ref{sde}).
When $\alpha\in (0,1)$, we will rather study the following equation:
for $M(dsdzdu)$ a Poisson measure on 
$[0,\infty)\times \rr_*\times \rr_*$ with intensity measure
$ds\; \nu_{a_-,a_+}^\alpha(dz)\; du$,
\begin{align}\label{sde2}
Y_t=&x+\int_0^t\! \intrs \! \intrs \! z [\indiq_{\{0<u < \gamma(Y_\sm)\}} 
- \indiq_{\{\gamma(Y_\sm)<u<0 \}}]
M(dsdzdu) 
+ \intot \! b(Y_s)ds,
\end{align}
where $\gamma(x)=\sg(\sigma(x)).|\sigma(x)|^\alpha$.
This equation is equivalent, in law, to (\ref{sde}). 
It has to be seen as another representation of (\ref{sde}).

\begin{lem}\label{eq}
Let $\alpha\in (0,1)$ and $a_-,a_+ \in [0,\infty)$.

(i) Let $(Y_t)_{t\geq 0}$ solve (\ref{sde2}). There exists a
$(\alpha,a_-,a_+)$-stable process $(Z_t)_{t\geq 0}$ such that
$(Y_t)_{t\geq 0}$ solves (\ref{sde}).

(ii) Let $(X_t)_{t\geq 0}$
solve (\ref{sde}). There exists, on an enlarged probability space,
a Poisson
measure $M$ on 
$[0,\infty)\times \rr_*\times \rr_*$ with intensity measure
$ds\; \nu_{a_-,a_+}^\alpha(dz)\; du$ such that
$(X_t)_{t\geq 0}$ solves 
(\ref{sde2}).
\end{lem}

Let us finally recall the following existence result.

\begin{prop}\label{weak}
Let $\alpha \in (0,2)\setminus\{1\}$ and $a_-,a_+ \in [0,\infty)$. 
Assume that $\sigma,b$ have at most linear growth.

(i) If $b,\sigma$ are continuous, there is weak existence for (\ref{sde}).

(ii) For any  solution to (\ref{sde}), any $\beta \in (0,\alpha)$, 
any $T>0$, $\E[\sup_{[0,T]} |X_t|^\beta]<\infty$.
\end{prop}

These results must be standard, but we found no precise reference.
The weak existence is almost contained in Situ \cite[Theorem 175]{s}.

\subsection{The case where $\alpha\in (1,2)$.}

This subsection is devoted to the study of (\ref{sde}) when $\alpha \in (1,2)$.
We first introduce some notation.

\begin{lem}\label{acos0}
For $\alpha \in (1,2)$, set $a=\cos(\pi \alpha) \in (-1,1)$.
Then for $c \in [0,1]$,
%(with $\arccos:[-1,1]\mapsto [0,\pi]$),
$$
\beta(\alpha,c):= \frac{1}{\pi}\arccos 
\left( \frac{c^2(1-a^2)-(1+ca)^2}{c^2(1-a^2)+(1+ca)^2}
\right) \in [\alpha-1,1].
$$
There holds $\beta(\alpha,0)=1$, $\beta(\alpha,1)=\alpha-1$
and $\beta(\alpha,c)\in(\alpha-1,1)$ for $c\in (0,1)$.
\end{lem}

We may assume that $a_-\leq a_+$ without loss of generality: if
$a_->a_+$, write $\sigma(X_\sm)dZ_s=\tilde\sigma(X_\sm)d \tilde Z_s$,
where $\tilde\sigma=-\sigma$ and $\tilde Z_t=-Z_t$ is a 
$(\alpha,a_+,a_-)$-stable process.

\begin{thm}\label{mr1}
Consider a stable process $(Z_t)_{t\geq 0}$ of order
$\alpha\in (1,2)$ with parameters $0\leq a_- \leq a_+$. 
Set $\beta=\beta(\alpha,a_-/a_+)$ as in Lemma \ref{acos0}.
Assume that $\sigma,b$ have at most linear growth and that for some
constants $\kappa_0,\kappa_1 \in [0,\infty)$,

$\bullet$ $\sigma$ is H\"older-continuous with index $(\alpha-\beta)/\alpha$
(which lies in $[1-1/\alpha,1/\alpha]$),

$\bullet$ for all $x,y\in \rr$,
$\sg(x-y)(a_+-a_-)(\sigma(y)-\sigma(x))\leq \kappa_1 |x-y|$,

$\bullet$ for all $x,y\in \rr$,
$\sg(x-y)(b(x)-b(y))\leq \kappa_0 |x-y|$.

Consider two solutions $(X_t)_{t\geq 0}$
and $(\tX_t)_{t\geq 0}$ to (\ref{sde}) started at $x$ and $\tx$.

(i) For any $t\geq 0$, there holds
\begin{align*}
\E\left[ |X_t-\tX_t|^{\beta}\right]\leq |x-\tx|^{\beta} 
e^{C t},
\end{align*}
where $C$ depends only on $\kappa_0,\kappa_1,\alpha,a_-,a_+$.
Thus pathwise uniqueness holds for (\ref{sde}).

(ii) If furthermore $b$ is constant and $(a_+-a_-)\sigma$ is non-decreasing,
then $\forall \; t\geq 0$, 
$$
\E\left[|X_t-\tX_t|^{\beta}\right]=|x-\tx|^{\beta}.
$$
\end{thm}

Observe that the condition on $b$ holds as soon as $b=b_1+b_2$,
with $b_1$ non-increasing and $b_2$ Lipschitz-continuous.
When $a_+=a_-$, we have $\beta=\alpha-1$ and thus we only assume that 
$\sigma$ is H\"older-continuous with index $1/\alpha$, as Komatsu \cite{k} 
or Bass \cite{b1}.
But when $a_-<a_+$, there is automatically a compensation in
the driving stable process, which introduces a sort of drift term.
Our assumption on $\sigma$ holds if $\sigma=\sigma_1+\sigma_2$,
with $\sigma_1$ Lipschitz-continuous and $\sigma_2$ H\"older-continuous
with index $(\alpha-\beta)/\alpha$ and non-decreasing.
Observe that $(\alpha-\beta)/\alpha<1/\alpha$, so that if $\sigma$
is non-decreasing, the assumption on $\sigma$ is weaker if $a_-<a_+$
than if $a_-=a_+$.
Finally, if $a_-=0$, then $\beta=1$, so that our assumption 
on $\sigma$ holds if $\sigma=\sigma_1+\sigma_2$,
with $\sigma_1$ Lipschitz-continuous and $\sigma_2$ H\"older-continuous
with index $1-1/\alpha$ and non-decreasing.

\vip

As compared to \cite{k,b1}, 
point (i) allows for a drift term,  allows us to treat the case $a_-\ne a_+$
and provides some stability with respect
to the initial datum.
Point (ii) is a remarkable property. It was already discovered
by Komatsu \cite{k} when $a_-=a_+$ (and thus $\beta=\alpha-1$), 
although not explicitly stated. A similar remarkable identity holds 
in the Brownian case (with $\alpha=2$ and $\beta=\alpha-1=1$), see Le Gall,
\cite[Theorem 1.3 and its proof]{lg}.

\vip

As a by-product, our proof allows us to check the following statement.
See \cite[Theorems 4 and 5]{fp} 
for similar considerations about the stochastic heat
equation.

\begin{prop}\label{inv}
Assume that $\alpha \in (1,2)$ and that $a_-=a_+>0$.
Suppose that $\sigma,b$ have at most linear growth, that 
$\sigma$ is H\"older-continuous with index $1/\alpha$ and that
$b$ is non-increasing and continuous.

(i) If $(b,\sigma)$ is injective, then (\ref{sde}) has at most one invariant
distribution.

(ii) If there is a strictly increasing function $\rho:\rr_+\mapsto \rr_+$
such that
$$
\forall \; x,y\in \rr,\quad 
\indiq_{\{x \ne y\}}
|x-y|^{\alpha-2}[|b(x)-b(y)|+|\sigma(x)-\sigma(y)|^\alpha]\geq \rho(|x-y|),
$$
then for any pair of solutions $(X_t)_{t\geq 0}$
and $(\tX_t)_{t\geq 0}$ to (\ref{sde}) started at $x$ and $\tx$
(driven by the same stable process $(Z_t)_{t\geq 0}$),
$\lim_{t\to\infty}|X_t-\tX_t|=0$ a.s.
\end{prop}

The basic example of application is the following: if 
$b(x)=-x$, then the conclusions of (i) and (ii) hold under
the sole assumption that $\sigma$ is H\"older-continuous with index
$1/\alpha$. In particular, no positivity of $\sigma$ is required at all.
We only treat the case where $a_-=a_+$, because the other possible results
are less interesting (although the proof is easily extended): 
some monotonicity conditions have to be imposed
on the {\it true drift coefficient}, which involves $b$ and $\sigma$.

\subsection{The case where $\alpha \in (0,1)$.}

Our goal is now to show that when $\alpha \in (0,1)$, 
(\ref{sde2}) is a {\it nice} representation
of (\ref{sde}), in the sense that pathwise uniqueness holds for a
larger class of functions $\sigma$, the Lipschitz condition being
replaced by a weaker condition. First, we state a general result
without monotonicity conditions on $\sigma$.

\begin{thm}\label{mr2}
Let $\alpha \in (0,1)$ and $a_-,a_+\in [0,\infty)$. 
Consider a Poisson measure $M$ on 
$[0,\infty)\times \rr_*\times \rr_*$ with intensity mesure
$ds\; \nu_{a_-,a_+}^\alpha(dz)\; du$.
Assume that $\sigma,b$ have at most linear growth and that
for some constant $\kappa_0\in [0,\infty)$, 

$\bullet$ $\gamma(x)=\sg(\sigma(x)).|\sigma(x)|^\alpha$ 
is H\"older-continuous with index $\alpha$,

$\bullet$ for all
$x,y\in \rr$,  $\sg(x-y)(b(x)-b(y))\leq\kappa_0|x-y|$.

Consider two solutions $(Y_t)_{t\geq 0}$
and $(\tY_t)_{t\geq 0}$ to (\ref{sde2}) started at $x$ and $\tx$.
Then for any $\beta \in (0,\alpha)$, any $t\geq 0$,
$$
\E\left[|Y_t-\tY_t|^\beta \right] \leq |x-\tx|^\beta e^{C t},
$$
where $C$ depends only on $\alpha,a_-,a_+,\beta,\kappa_0$ and on the
H\"older constant of $\gamma$. Thus pathwise uniqueness holds
for (\ref{sde2}).
\end{thm}

Observe at once that if $\sigma$ is bounded below by a positive
constant and H\"older-continuous
with index $\alpha$, then $\gamma$ is also
H\"older-continuous with index $\alpha$. But if 
$\sigma$ vanishes, it has to be Lipschitz-continuous around its zeros.
This is not only a technical condition as shown by 
Komatsu \cite{k} or Bass \cite[Remark 3.4]{b1}: if $\alpha \in (0,1)$, 
$x=0$, $b=0$, $a_-=a_+=1$ and $\sigma(x)=|x|^\beta$ (whence
$\gamma(x)=|x|^{\beta\alpha}$)
for some $\beta<1$, then uniqueness in law fails for (\ref{sde}), whence
it also fails for (\ref{sde2}).

\vip

It might be surprising at first glance that in some cases,
pathwise uniqueness holds for (\ref{sde2}) but not for (\ref{sde}).
This comes from the fact that, e.g. when starting from two
initial positions $x$ and $\tx$,  (\ref{sde2})
builds two different stable processes (coupled in a suitable way)
to drive $(Y_t)_{t\geq 0}$ and $(\tY_t)_{t\geq 0}$,
while in (\ref{sde}), the same stable process 
drives $(X_t)_{t\geq 0}$ and $(\tX_t)_{t\geq 0}$. 
We see that the choice made
in (\ref{sde2}) is more efficient.

\vip

Let us now try to take advantage
of some monotonicity considerations when $a_-\ne a_+$.
This seems possible only if $\alpha \in (1/2,1)$
and if $a_-/a_+$ is small enough.

\begin{lem}\label{acos}
For $\alpha \in (1/2,1)$, set $a=\cos(\pi \alpha) \in (-1,0)$. Then for 
$c \in [0,-a)$,
%(with $\arccos:[-1,1]\mapsto [0,\pi]$),
$$
\beta(\alpha,c):= \frac{1}{\pi}\arccos 
\left( \frac{1-a^2-(c+a)^2}{1-a^2+(c+a)^2}
\right) \in (0,2\alpha-1].
$$
There holds $\beta(\alpha,0)=2\alpha-1$ and, 
for any $c\in (0,1)$, $\lim_{\alpha \to 1-} \beta(\alpha,c)=1$.
\end{lem}

We only consider the case $a_-<a_+$ without loss of generality.

\begin{thm}\label{mrsuper}
Assume that $\alpha \in (1/2,1)$,
that $a_-/a_+ < |\cos(\pi\alpha)|$ and set 
$\beta:=\beta(\alpha,a_-/a_+)$ as in Lemma \ref{acos}.
Consider a  Poisson measure $M$ on 
$[0,\infty)\times \rr_*\times \rr_*$ with intensity measure
$ds\; \nu_{a_-,a_+}^\alpha(dz)\; du$.
Assume that $\sigma,b$ have at most linear growth and that
for some constants $\kappa_0,\kappa_1 \in [0,\infty)$,

$\bullet$ $\gamma(x)=\sg(\sigma(x)).|\sigma(x)|^\alpha$ 
is H\"older-continuous with index $\alpha-\beta$,

$\bullet$ for all
$x,y\in \rr$,  $\sg(x-y)(\gamma(x)-\gamma(y))\leq\kappa_1|x-y|^\alpha$,

$\bullet$ for all
$x,y\in \rr$,  $\sg(x-y)(b(x)-b(y))\leq\kappa_0|x-y|$.

Consider two solutions $(Y_t)_{t\geq 0}$
and $(\tY_t)_{t\geq 0}$ to (\ref{sde2}) started at $x$ and $\tx$.

(i) Then for any  $t\geq 0$,
$$
\E\left[|Y_t-\tY_t|^{\beta} \right] 
\leq |x-\tx|^{\beta} e^{C t},
$$
where $C$ depends only on $a_-,a_+,\alpha,\kappa_0,\kappa_1$. 
Thus pathwise uniqueness holds for (\ref{sde2}).

(ii) If furthermore $b$ is constant and
$\gamma$ is non-increasing,
then $\forall \; t\geq 0$,
$$
\E\left[|Y_t-\tY_t|^{\beta}\right]
= |x-\tx|^{\beta}.
$$
\end{thm}

This last property is of course remarkable. 
If $a_-=0$, the above result holds when $\gamma=\gamma_1+\gamma_2$,
with $\gamma_1$ H\"older-continuous 
with index $\alpha$ and $\gamma_2$ non-increasing
and H\"older-continuous with index $1-\alpha$, which is very small when
$\alpha$ is close to $1$. More generally, 
when $a_-<a_+$ and if $\alpha$ is very close to $1$, one has
to assume only very few regularity on $\gamma$, provided it is
non-increasing.

\subsection{Comments}\label{comments}

First observe that when $a_-<a_+$, the favorable monotonicity of $\sigma$
is not the same if $\alpha \in (0,1)$ and if $\alpha \in (1,2)$.
This is due to the fact that when $\alpha \in (1,2)$, the main problem
is due to the compensation (which appears negatively in the equation).

\vip

\begin{figure}
\includegraphics[width=11cm]{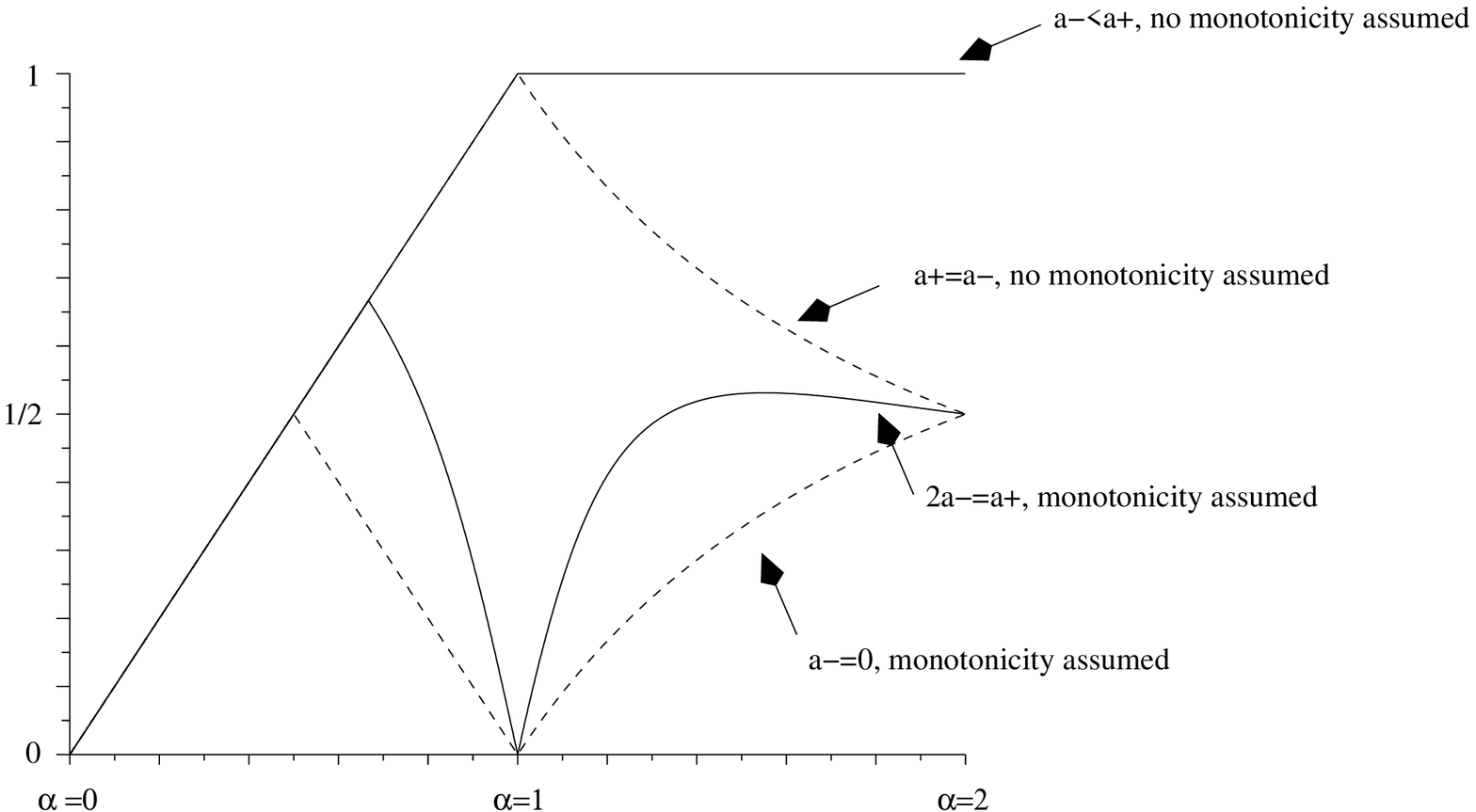}
\vip
\parbox{12cm}{\footnotesize{Figure 1. 
Index of H\"older regularity of $\gamma$ (if $\alpha \in (0,1)$)
or $\sigma$ (if $\alpha \in (1,2)$)
required for pathwise uniqueness as a function of $\alpha$.
The four curves coincide on $[0,1/2]$.}
}
\end{figure}

Let us summarize roughly our results. Denote by $H(\delta)$ the set
of H\"older-continuous functions with index $\delta$ and by $ H^{\downarrow}
(\delta)$ (resp. $ H^{\uparrow}(\delta)$) its subset of non-increasing
(resp. non-decreasing) functions. Recall that when $\sigma$ is bounded
below by a positive constant, the regularity of $\gamma(x)=\sg(\sigma(x)).
|\sigma(x)|^\alpha$ is the same as that of $\sigma$.
We have pathwise uniqueness for (\ref{sde2}) (if $\alpha \in (0,1)$)
and (\ref{sde}) (if $\alpha \in (1,2)$) if $b=b_1+b_2$ has at most linear
growth, with
$b_1 \in H(1)$ and $b_2$ non-increasing and if $\sigma=\sigma_1+\sigma_2$
(or $\gamma=\gamma_1+\gamma_2$) satisfies (we set $\beta(\alpha,c)=0$
if $\alpha \in (0,1/2]$ or if $c\geq -\cos(\pi\alpha)$):

\vip
\begin{center}
\begin{tabular}{|c||c|c|}
\hline
& $\alpha \in (0,1)$ & $\alpha \in (1,2)$   
\\ \hline \hline
$a_-=a_+$ & $\gamma \in H(\alpha)$ 
& $\sigma \in H(1/\alpha)$ 
\\ \hline 
$a_-<a_+$ & $ \begin{matrix} \gamma_1 \in H(\alpha),
\\
\gamma_2 \in H^{\downarrow}(\alpha-\beta(\alpha,a_-/a_+))\end{matrix}$
 & $ \begin{matrix} \sigma_1 \in H(1),\\
\sigma_2 \in H^{\uparrow}(1-\beta(\alpha,a_-/a_+)/\alpha)\end{matrix}$
\\ \hline 
$a_-=0$ & $ \begin{matrix} \gamma_1 \in H(\alpha),
\\
\gamma_2 \in H^{\downarrow}(1-\alpha)\end{matrix}$
 & $ \begin{matrix} \sigma_1 \in H(1),\\
\sigma_2 \in H^{\uparrow}(1-1/\alpha)\end{matrix}$
\\ \hline 
\end{tabular}
\end{center}
\vip
Thus the situation is quite intricate.
When $a_-=a_+$ and $\sigma$ is bounded from below, we have to assume that 
$\sigma\in H(\min\{\alpha,1/\alpha\})$.
It seems quite strange that the
required regularity of $\sigma$ is low when $\alpha$ is small, maximal
when $\alpha=1$ and small again when $\alpha$ is near $2$. 
A more tricky representation of (\ref{sde}) 
might allow one to obtain some better results.

\vip

When $a_-=0$ and $\sigma$ is bounded from below and monotonic, 
we have to assume that $\sigma \in H(\alpha)$ (if $\alpha \in (0,1/2]$),
$\sigma \in H^{\downarrow}(1-\alpha)$ (if $\alpha \in (1/2,1)$)
and $\sigma \in H^{\uparrow}(1-1/\alpha)$ (if $\alpha \in (1,2)$). 
Thus few regularity is needed
when $\alpha$ is near $0$ or $1$ and higher regularity is needed
when $\alpha$ is near $1/2$ and $2$.

\vip

Theorem \ref{mr1} is not so good when $a_-< a_+$, because we have
to assume the Lipschitz-continuity of the decreasing
part of $\sigma$. On the contrary, Theorem \ref{mr2} works quite well
for any value of $a_-,a_+$.

\vip

Theorems \ref{mr1} and \ref{mrsuper} really rely on specific
properties of stable processes.
Theorem \ref{mr2}, of which the proof is much simpler, 
may be easily extended to other 
jumping S.D.E.s with finite variations. For example, some of the 
Lipschitz assumptions
of \cite{f} can be consequently weakened.

\subsection{Plan of the paper} In the next section, we prove Lemmas \ref{acos0}
and \ref{acos} and show that some integrals vanish. These
integrals are those that appear when we use the It\^o
formula to compute $|X_t-\tX_t|^\beta$ for two solutions to
(\ref{sde}) or (\ref{sde2}).
Section \ref{hc} shows how to approximate these integrals.
We prove Theorem \ref{mr1} and Proposition \ref{inv} in Section \ref{p1}.
Section \ref{p2} is devoted to the proofs of Theorems \ref{mr2} and
\ref{mrsuper}.
We finally check Proposition \ref{weak} and Lemma \ref{eq} in Section \ref{ch}.

\section{Computation of some integrals}\label{integ}
This technical section contains the main tools of the paper.
We introduce
\begin{align}
&\hbox{for } \alpha\in (0,1) \hbox{ and } \beta\in (0,\alpha), \quad
I_{a_-,a_+}^{\alpha,\beta} = \intrs \left[|1-x|^\beta-1 \right] 
\nu_{a_-,a_+}^\alpha(dz)
, \label{dfI}\\
&\hbox{for } \alpha\in (1,2) \hbox{ and } \beta\in (0,\alpha), \quad
\tI_{a_-,a_+}^{\alpha,\beta} = \intrs \left[|1+x|^\beta-1 -\beta x\right] 
\nu_{a_-,a_+}^\alpha(dz). \label{dftI}
\end{align}
Observe that all the above integrals converge absolutely.
The aim of this section is to prove Lemmas \ref{acos0} and \ref{acos}, as
well as the following identities.

\begin{lem}\label{basic}
(i) Let $\alpha\in (1/2,1)$ and $0\leq a_-\leq a_+$ such that
$a_-/a_+ < -\cos (\pi\alpha)$. Set $\beta=\beta(\alpha,a_-/a_+)
\in(0,2\alpha-1]$ 
as in Lemma
\ref{acos}. There holds $I^{\alpha,\beta}_{a_-,a_+}=0$.

(ii) Let $\alpha \in (1,2)$ and  $0\leq a_-\leq a_+$.
Set $\beta=\beta(\alpha,a_-/a_+) \in [\alpha-1,1]$ as in Lemma
\ref{acos0}. There holds $\tI^{\alpha,\beta}_{a_-,a_+}=0$.
\end{lem}

\begin{proof}
We start with point (i). Observe that $\beta\leq 2\alpha-1<\alpha$, so that
the integral is convergent.
We write $I_{a_-,a_+}^{\alpha,\beta}=a_- A_1+a_+A_2+a_+A_3$, 
where \begin{align*}
A_1=&\int_0^\infty [(1+x)^{\beta}-1]x^{-\alpha-1}dx,\\
A_2=&\int_0^1 [(1-x)^{\beta}-1] x^{-\alpha-1}dx,\\
A_3=&\int_1^\infty [(x-1)^{\beta}-1] x^{-\alpha-1}dx.
\end{align*}
Using an integration by parts and then putting $u=1/(1+x)$, one can check that
\begin{align*}
A_1=\frac{\beta}{\alpha}\int_0^\infty (1+x)^{\beta-1}x^{-\alpha}dx
=\frac{\beta}{\alpha}\int_0^1 u^{\alpha-\beta-1}(1-u)^{-\alpha}du
=\frac{\beta \Gamma(\alpha-\beta)\Gamma(1-\alpha)}{\alpha \Gamma(1-\beta)}
\end{align*}
where $\Gamma$ is the Euler function. Next, an integration by parts implies
that
$$
A_2=\frac 1 \alpha - \frac{\beta}{\alpha} \int_0^1 (1-x)^{\beta-1}x^{-\alpha} dx
= \frac 1 \alpha - \frac{\beta \Gamma(\beta)\Gamma(1-\alpha)}
{\alpha \Gamma(1-(\alpha-\beta))}.
$$
Finally, setting $x=1/u$, we get
$$
A_3=\int_1^\infty \!\! (x-1)^{\beta}x^{-\alpha-1}dx - \frac{1}{\alpha}
= \int_0^1 (1-u)^\beta u^{\alpha-\beta-1} du - \frac{1}{\alpha}
=  \frac{\Gamma(\beta+1)\Gamma(\alpha-\beta)}{\Gamma(\alpha+1)}
- \frac{1}{\alpha}.
$$
We thus find, recalling that $\Gamma(a+1)=a\Gamma(a)$,
$$
I^{\alpha,\beta}_{a_-,a_+} = \frac{\beta}{\alpha}\left[
a_- \frac{\Gamma(\alpha-\beta)\Gamma(1-\alpha)}{\Gamma(1-\beta)}
-a_+ \frac{\Gamma(\beta)\Gamma(1-\alpha)}{\Gamma(1-(\alpha-\beta))}
+a_+ \frac{\Gamma(\beta)\Gamma(\alpha-\beta)}{\Gamma(\alpha)}
\right].
$$
Using now Euler's reflection formula 
$\Gamma(x)\Gamma(1-x)=\pi/\sin(\pi x)$ for $x\in (0,1)$,
\begin{align*}
I^{\alpha,\beta}_{a_-,a_+} =& 
\frac{\beta\Gamma(\beta)\Gamma(\alpha-\beta)}{\alpha\Gamma(\alpha)}
\left[
a_- \frac{\sin (\pi\beta)}{\sin(\pi\alpha)}
-a_+ \frac{\sin (\pi(\alpha-\beta))}{\sin(\pi\alpha)}
+a_+
\right]\\
=&\frac{a_+\beta\Gamma(\beta)\Gamma(\alpha-\beta)}
{\alpha\Gamma(\alpha)\sin(\pi\alpha)}
\left[ c \sin (\pi\beta) - \sin (\pi(\alpha-\beta)) +\sin(\pi\alpha)\right]
\end{align*}
where we have set $c=a_-/a_+$. We have chosen $\beta=\beta(\alpha,c)$
in such a way that  
$c \sin (\pi\beta) - \sin (\pi(\alpha-\beta)) +\sin(\pi\alpha)=0$, whence
$I^{\alpha,\beta}_{a_-,a_+}=0$ as desired. Indeed, recall that 
$\cos(\pi\beta)=b$,
where $b=(1-a^2-(c+a)^2)/(1-a^2+(c+a)^2)$, with $a=\cos(\pi\alpha)$. 
Since $\beta,\alpha \in (0,1)$, we have
$\sin(\pi\alpha)=\sqrt{1-a^2}$ and 
$\sin(\pi\beta)=\sqrt{1-b^2}$, whence
\begin{align*}
&c \sin (\pi\beta) - \sin (\pi(\alpha-\beta)) +\sin(\pi\alpha)\\
=&c \sin (\pi\beta) - \sin (\pi\alpha)\cos(\pi\beta) + 
\sin (\pi\beta)\cos(\pi\alpha) +\sin(\pi\alpha)\\
=&(a+c)\sqrt{1-b^2} +(1-b) \sqrt{1-a^2}.
\end{align*}
Recall that $a+c<0<1-b$, since $c=a_-/a_+<-\cos(\pi\alpha)=-a$. 
We thus need to check that
$(a+c)^2(1-b^2)=(1-b)^2(1-a^2)$,
i.e. $(a+c)^2(1+b)=(1-b)(1-a^2)$. This is easily verified.

\vip

We now prove (ii). We write $\tI_{a_-,a_+}^{\alpha,\beta}=a_+B_1+a_-B_2+a_-B_3$, 
where
\begin{align*}
B_1=& \int_0^\infty [(1+x)^\beta-1-\beta x]x^{-\alpha-1}dx,\\
B_2=& \int_1^\infty [(x-1)^\beta-1+\beta x]x^{-\alpha-1}dx,\\
B_3=& \int_0^1 [(1-x)^\beta-1+\beta x]x^{-\alpha-1}dx.
\end{align*}
Using two integrations by parts and then putting $u=1/(1+x)$, 
one can prove that, if $\alpha-1\leq \beta < 1$,
\begin{align*}
B_1 =& \frac{\beta (\beta-1)}{\alpha(\alpha-1)}\int_0^\infty 
(1+x)^{\beta-2}x^{1-\alpha}dx \\
=&  \frac{\beta (\beta-1)}{\alpha(\alpha-1)}
\int_0^1 u^{\alpha-\beta-1}(1-u)^{1-\alpha}du\\
=& - \frac{\beta (1-\beta) \Gamma(2-\alpha)\Gamma(\alpha-\beta)}
{\alpha(\alpha-1)\Gamma(2-\beta)}. \nonumber
\end{align*}
Since now $\alpha\in (1,2)$ and $\beta \in (0,1)$,
\begin{align}
\Gamma(2-\alpha)=&\frac{\pi}{\Gamma(\alpha-1)
\sin(\pi(\alpha-1))}=\frac{\pi(\alpha-1)}{\Gamma(\alpha)\sin(\pi(\alpha-1))},
\label{eulerr} \\
\Gamma(2-\beta)=&(1-\beta)\Gamma(1-\beta)=
\frac{\pi(1-\beta)}{\Gamma(\beta)\sin(\pi\beta)}.\nonumber
\end{align}
Hence
\begin{align*}
B_1=- \frac{\beta \Gamma(\beta)\Gamma(\alpha-\beta) \sin(\pi\beta)}
{\alpha \Gamma(\alpha) \sin(\pi(\alpha-1))}.
\end{align*}
This formula remains valid if $\beta=1$, since then $B_1=0$ and 
$\sin(\pi\beta)=0$.
Next we use one integration by parts and we put $u=1/x$ to get
\begin{align*}
B_2 =& \frac{\beta-1}{\alpha} +\frac{\beta}{\alpha}\int_1^\infty 
[1+(x-1)^{\beta-1}]x^{-\alpha}dx \\
=&  \frac{\beta-1}{\alpha} 
+\frac{\beta}{\alpha(\alpha-1)} + \frac{\beta}{\alpha}\int_0^1
u^{\alpha-\beta-1}(1-u)^{\beta-1}du\\
=& -\frac{1}{\alpha}+\frac{\beta}{\alpha-1}+ 
\frac{\beta \Gamma(\beta)\Gamma(\alpha-\beta)}{\alpha\Gamma(\alpha)}.
\end{align*}
Finally, an integration by parts shows that 
$$
B_3=\frac{1-\beta}{\alpha} - \frac{\beta}{\alpha} \int_0^1 
[(1-x)^{\beta-1}-1] x^{-\alpha}dx= 
\frac{1-\beta}{\alpha} - \frac{\beta}{\alpha} [G_1+G_2],
$$
where
$$
G_1= \int_0^1 
[(1-x)^{\beta-1}-(1-x)^\beta] x^{-\alpha}dx= \int_0^1 
(1-x)^{\beta-1}x^{1-\alpha}dx=\frac{\Gamma(\beta)\Gamma(2-\alpha)}
{\Gamma(\beta+2-\alpha)}
$$
and,  using an integration by parts, 
\begin{align*}
G_2=&\int_0^1 
[(1-x)^\beta-1] x^{-\alpha}dx \\
=& 
\frac{1}{\alpha-1}
-\frac{\beta}{\alpha-1} \int_0^1 (1-x)^{\beta-1}x^{1-\alpha}dx\\
=&\frac{1}{\alpha-1}- \frac{\beta \Gamma(\beta)\Gamma(2-\alpha)}
{(\alpha-1)\Gamma(\beta+2-\alpha)}.
\end{align*} 
Thus
\begin{align*}
B_3=&\frac 1 \alpha - \frac \beta {\alpha-1} + 
\frac{\beta(\beta+1-\alpha)\Gamma(\beta)\Gamma(2-\alpha)}
{\alpha(\alpha-1)\Gamma(\beta+2-\alpha)}\\
=&\frac 1 \alpha - \frac \beta {\alpha-1} + 
\frac{\beta\Gamma(\beta)\Gamma(\alpha-\beta) \sin(\pi(\alpha-\beta))}
{\alpha \Gamma(\alpha)\sin(\pi(\alpha-1))}.
\end{align*}
We used (\ref{eulerr}) and that 
$$
\frac{\Gamma(\beta+2-\alpha)}{\beta+1-\alpha}=\Gamma(\beta+1-\alpha)
=\frac{\pi}{\Gamma(\alpha-\beta) \sin(\pi(\alpha-\beta))}.
$$
This last equality uses that $\beta-\alpha+1 \in (0,1)$, but
one easily checks that the expression of $B_3$ 
remains valid if $\beta=\alpha-1$, because then 
$1+\beta-\alpha=\sin(\pi(\alpha-\beta))=0$. We finally find that
\begin{align*}
\tI_{a_-,a_+}^{\alpha,\beta}=& 
\frac{\beta\Gamma(\beta)\Gamma(\alpha-\beta)}{\alpha\Gamma(\alpha)}
\left[- a_+ \frac{\sin(\pi\beta)}{\sin(\pi(\alpha-1))}  
+a_- +a_- \frac{\sin(\pi(\alpha-\beta))}{\sin(\pi(\alpha-1))}
\right].
\end{align*}
Set $c=a_-/a_+$ and recall that
$b:=\cos(\pi\beta)=(c^2(1-a^2)-(1+ca)^2)/(c^2(1-a^2)+(1+ca)^2)$,
for $a=\cos(\pi\alpha)\in (-1,1)$. It remains to check that
$\sin(\pi\beta)=c\sin(\pi(\alpha-1))
+ c\sin(\pi(\alpha-\beta))$, i.e. 
$\sin(\pi\beta)=-c\sin(\pi\alpha)+c\sin(\pi\alpha)\cos(\pi\beta)
-c \sin(\pi\beta)\cos (\pi\alpha)$. 
Observe that since $\alpha\in(1,2)$,
$\sin(\pi\alpha)=-\sqrt{1-a^2}$, while since $\beta \in (0,1]$,
$\sin(\pi\beta)=\sqrt{1-b^2}$. We need to verify that
$\sqrt{1-b^2}=c\sqrt{1-a^2}-cb\sqrt{1-a^2}-c a \sqrt{1-b^2}$, i.e. that
$(1+ac)\sqrt{1-b^2}=c(1-b)\sqrt{1-a^2}$, i.e. that
$(1+ac)^2(1+b)=c^2(1-a^2)(1-b)$. This is easily done.
\end{proof}

We now give the

\begin{preuve} {\it of Lemma \ref{acos0}.}
Recall that $\alpha\in (1,2)$, that $a=\cos(\pi\alpha)\in (-1,1)$, 
that $0\leq c \leq 1$
and that $\beta(\alpha,c)=\pi^{-1}\arccos b$, where
$b:=(c^2(1-a^2)-(1+ca)^2)/(c^2(1-a^2)+(1+ca)^2)$.

First, we have $\beta(\alpha,0)=\pi^{-1}\arccos (-1) =1$ and
$\beta(\alpha,1)=\pi^{-1}\arccos (-a) =\pi^{-1}\arccos (-\cos(\pi\alpha))=
\alpha-1$. 

To check that  $\beta(\alpha,c) \in  (\alpha-1,1)$ 
if $c\in (0,1)$, it suffices
to prove that $b \in (-1,-a)$ (because $-1=\cos(\pi)$ and 
$-a=\cos(\pi(\alpha-1))$).
First, $b>-1$ is obvious if $c>0$. Next, we 
have to check that $c^2(1-a^2)-(1+ca)^2<-a(c^2(1-a^2)+(1+ca)^2)$,
i.e. that $c^2(1+a)^2<(1+ca)^2$, which holds true because $c<1$ and $|a|< 1$.
\end{preuve}

We conclude this section with the

\begin{preuve} {\it of Lemma \ref{acos}.} Recall that $\alpha \in (1/2,1)$,
that $0\leq c<-a=-\cos(\pi\alpha)<1$ 
and that $\beta(\alpha,c)=\pi^{-1}\arccos b$,
where $b=(1-a^2-(c+a)^2)/(1-a^2+(c+a)^2)$.
First, $\beta(\alpha,0)=\pi^{-1}\arccos(1-2a^2)=
\pi^{-1}\arccos(1-2\cos^2(\pi\alpha))= 2\alpha-1$.

To prove that
$\beta(\alpha,c)\in (0,2\alpha-1]$, it suffices to check that
$b \in [1-2a^2,1)$. First, $b<1$ is obvious. Next, 
$b\geq1-2a^2$ because $(1-a^2)-(c+a)^2\geq(1-2a^2)[(1-a^2)+(c+a)^2]$,
since $2a^2(1-a^2)\geq2(c+a)^2(1-a^2)$. Indeed, $a^2>(c+a)^2$,
since $0\leq c<-a$.

Finally, for $c<1$ fixed, $\lim_{\alpha \to 1-} a= -1$, whence 
$\lim_{\alpha \to 1-} b= -1$ and thus
$\lim_{\alpha \to 1-} \beta(\alpha,c) = \pi^{-1}\arccos(-1)=1$.
\end{preuve}

\section{Approximation lemmas}\label{hc}

To prove our main results, we will apply It\^o's formula
to compute $|X_t-\tX_t|^\beta$, for $(X_t)_{t\geq 0}$ and $(\tX_t)_{t\geq 0}$
two solutions to (\ref{sde}) or (\ref{sde2}), with some suitable
value of $\beta \in (0,\alpha)$. This is not licit, since the function
$|x|^\beta$ is not of class $C^2$. The two lemmas below will allow us
to overcome this difficulty.

\begin{lem}\label{hyperchiant2}
Let $0<\beta<\alpha<1$ and $a_-,a_+ \in [0,\infty)$. 
For $\eta>0$, set $\phi_\eta(x)=(\eta^2+x^2)^{\beta/2}$. 
For $\Delta\in \rr_*$,
\begin{align*}
J_{a_-,a_+}^{\alpha,\beta,\eta}(\Delta)
:=&\intrs \left[\phi_\eta(\Delta-z)-\phi_\eta(\Delta) 
\right] \nu_{a_-,a_+}^\alpha(dz) \\
\to& |\Delta|^{\beta-\alpha}\left[\indiq_{\Delta>0} 
I_{a_-,a_+}^{\alpha,\beta}
+  \indiq_{\Delta<0} I_{a_+,a_-}^{\alpha,\beta}\right]
\end{align*}
as $\eta \to 0$, recall (\ref{dfI}). Furthermore,
for all $\eta>0$, all $\Delta\in \rr_*$,
$$
|J_{a_-,a_+}^{\alpha,\beta,\eta}(\Delta)|\leq 
K_{a_-,a_+}^{\alpha,\beta,\eta}(\Delta):=\intrs 
\left|\phi_\eta(\Delta-z)-\phi_\eta(\Delta) 
\right| \nu_{a_-,a_+}^\alpha(dz)\leq C|\Delta|^{\beta-\alpha},
$$
where $C$ depends 
only on $\alpha,a_-,a_+,\beta$.
\end{lem}

\begin{proof}
We fix $\Delta\in \rr_*$ and we observe that for all $\eta>0$,
\begin{align}\label{ccc}
|\phi_\eta(\Delta-z)-\phi_\eta(\Delta)| \leq C \min \left\{|z|^{\beta},
|\Delta|^{\beta-1} |z| \right\}.
\end{align}
This is easily deduced from the facts that 
$|\phi_\eta(x+y)-\phi_\eta(x)|\leq |y|^\beta$ and
$|\phi_\eta'(x)|\leq \beta |x|^{\beta-1}$. Separate the cases 
$|z|\leq |\Delta|/2$
and $|z|\geq |\Delta|/2$. But now
\begin{align*}
\intrs  \min\{|z|^{\beta},|\Delta|^{\beta-1}|z| \} |z|^{-\alpha-1}dz 
\leq& \int_{|z|\geq |\Delta|}|z|^{\beta-\alpha-1}dz+  \int_{|z|\leq |\Delta|} 
|\Delta|^{\beta-1} |z|^{-\alpha}dz \\
=& C |\Delta|^{\beta-\alpha}. 
\end{align*}
We immediately deduce that $|K_{a_-,a_+}^{\alpha,\beta,\eta}(\Delta)|\leq 
C|\Delta|^{\beta-\alpha}$. And by Lebesgue's dominated convergence theorem,
since $\lim_{\eta\to 0} \phi_\eta(x)=|x|^\beta$ for all $x\in \rr$,
\begin{align*}
\lim_{\eta \to 0} J_{a_-,a_+}^{\alpha,\beta,\eta}(\Delta)=&\intrs
\left[|\Delta-z|^{\beta} - |\Delta|^{\beta} \right] \nu_{a_-,a_+}(dz)\\
=&|\Delta|^{\beta} \intrs \left[|1-z/\Delta|^{\beta} - 1 \right] \nu_{a_-,a_+}(dz)
\\
=& \indiq_{\{\Delta>0\}}|\Delta|^{\beta} 
\intrs \left[|1-z/|\Delta||^{\beta} - 1 \right] \nu_{a_-,a_+}(dz)
\\
&+  \indiq_{\{\Delta<0\}}|\Delta|^{\beta} 
\intrs \left[|1-z/|\Delta||^{\beta} - 1 \right] \nu_{a_+,a_-}(dz).
\end{align*}
In the last inequality and when $\Delta<0$, we have used the 
substitution $x=-z$, which leads to 
$\nu_{a_-,a_+}^\alpha(dz)=\nu_{a_+,a_-}^\alpha(dx)$.
Using finally the substitution $x=z/|\Delta|$, for which
$\nu_{a_-,a_+}^\alpha(dz)=|\Delta|^{-\alpha} \nu_{a_-,a_+}^\alpha(dx)$, we get
\begin{align*}
\lim_{\eta \to 0} J_{a_-,a_+}^{\alpha,\beta,\eta}(\Delta)
=&\indiq_{\{\Delta>0\}}|\Delta|^{\beta-\alpha} 
\intrs \left[|1-x|^{\beta} - 1 \right] \nu_{a_-,a_+}(dx)\\
&+  \indiq_{\{\Delta<0\}}|\Delta|^{\beta-\alpha} 
\intrs \left[|1-x|^{\beta} - 1 \right] \nu_{a_+,a_-}(dx)
\end{align*}
as desired.
\end{proof}

\begin{lem}\label{hyperchiant}
Let $0<\beta\leq 1<\alpha<2$ and $a_-,a_+$ in $[0,\infty)$.
For $\eta>0$ and $x\in \rr$, set $\phi_\eta(x)=(\eta^2+x^2)^{\beta/2}$.
Define, for $\Delta,\delta \in \rr$,
\begin{align*}
\tJ_{a_-,a_+}^{\alpha,\beta,\eta}(\Delta,\delta)
:=&\intrs \left\{\phi_\eta(\Delta+\delta z) -\phi_\eta(\Delta)   
-\delta z \phi_\eta'(\Delta) \right\} \nu_{a_-,a_+}^\alpha(dz),\\
\tK_{a_-,a_+}^{\alpha,\beta,\eta}(\Delta,\delta):=
&\int_{|\delta z|\leq |\Delta|} \left(\phi_\eta(\Delta+\delta z) 
-\phi_\eta(\Delta)   -|\Delta+\delta z|^{\beta} + |\Delta|^{\beta}
\right)^2 \nu_{a_-,a_+}^\alpha(dz),\\
\tL_{a_-,a_+}^{\alpha,\beta,\eta}(\Delta,\delta):=&\int_{|\delta z|> |\Delta|} 
\left|\phi_\eta(\Delta+\delta z) 
-\phi_\eta(\Delta)   -|\Delta+\delta z|^{\beta} + |\Delta|^{\beta}
\right|
\nu_{a_-,a_+}^\alpha(dz).
\end{align*}
For any $\Delta\in\rr_*$, any $\delta\in \rr$,
\begin{align*}
&\lim_{\eta \to 0} \tJ_{a_-,a_+}^{\alpha,\beta,\eta}(\Delta,\delta)
=|\Delta|^{\beta-\alpha}|\delta|^\alpha
\left[\indiq_{\{\delta\Delta>0\}}\tI^{\alpha,\beta}_{a_-,a_+} 
+\indiq_{\{\delta\Delta<0\}}
\tI^{\alpha,\beta}_{a_+,a_-}   \right],\\
&\lim_{\eta \to 0} \tK_{a_-,a_+}^{\alpha,\beta,\eta}(\Delta,\delta)= 
\lim_{\eta \to 0} \tL_{a_-,a_+}^{\alpha,\beta,\eta}(\Delta,\delta) = 0.
\end{align*}
Furthermore, we can find a constant $C$, depending 
only on $\alpha,a_-,a_+,\beta$,
such that for all $\eta>0$, all $\Delta\in \rr_*$, all $\delta \in \rr$,
\begin{align*}
&|\tJ_{a_-,a_+}^{\alpha,\beta,\eta}(\Delta,\delta)|
\leq C|\Delta|^{\beta-\alpha}|\delta|^\alpha,\\
&\tK_{a_-,a_+}^{\alpha,\beta,\eta}(\Delta,\delta)
\leq C |\Delta|^{2\beta-\alpha}|\delta|^\alpha, \\ 
&\tL_{a_-,a_+}^{\alpha,\beta}(\Delta,\delta)
\leq C |\Delta|^{\beta-\alpha}|\delta|^\alpha.
\end{align*}
\end{lem}

\begin{proof}  We first observe that there is a constant $C$ such that
for all $\eta>0$,
\begin{align*}
\left|\phi_\eta(\Delta+\delta z) - \phi_\eta(\Delta)  \right| \leq &
C  \min \left\{|\Delta|^{\beta-1} |\delta z|, |\delta z|^\beta\right\},
\\
\left|\phi_\eta(\Delta+\delta z) - \phi_\eta(\Delta)   
-\delta z \phi_\eta'(\Delta) \right|
\leq& C \min \left\{|\Delta|^{\beta-2} (\delta z)^2,|\Delta|^{\beta-1} 
|\delta z|\right\}.
\end{align*}
This is easily deduced from the facts that 
$|\phi_\eta(x+y) - \phi_\eta(x)| \leq |y|^{\beta}$, 
$|\phi_\eta'(x)|\leq \beta|x|^{\beta-1}$,
$|\phi_\eta''(x)|\leq C |x|^{\beta-2}$ and $\beta-1\leq 0$. 
Separate the cases $|\delta z|\leq |\Delta|/2$ and  $|\delta z|\geq |\Delta|/2$.
Similarly,
\begin{align*}
\left| |\Delta+\delta z|^\beta - |\Delta|^\beta  \right| \leq &
C  \min \left\{|\Delta|^{\beta-1} |\delta z|, |\delta z|^\beta\right\}.
\end{align*}

\vip

Next we observe that 
\begin{align*}
&\intrs  \min \left\{|\Delta|^{\beta-2} (\delta z)^2,|\Delta|^{\beta-1} 
|\delta z|\right\}     |z|^{-\alpha-1}dz \\
\leq& |\Delta|^{\beta-2} |\delta|^2 \int_{\{|z|\leq|\Delta|/|\delta|\}}  
|z|^{1-\alpha}dz
+  |\Delta|^{\beta-1} |\delta| \int_{\{|z|\geq|\Delta|/|\delta|\}}  |z|^{-\alpha}dz 
= C  |\Delta|^{\beta-\alpha}|\delta|^\alpha,
\end{align*}
from which we immediately deduce that 
$|\tJ_{a_-,a_+}^{\alpha,\beta,\eta}(\Delta,\delta)|
\leq C|\Delta|^{\beta-\alpha}|\delta|^\alpha$ and that we can apply
Lebesgue's dominated convergence theorem:
\begin{align*}
\lim_{\eta \to 0} \tJ_{a_-,a_+}^{\alpha,\beta,\eta}(\Delta,\delta)
=&\intrs
\left\{|\Delta+\delta z|^{\beta} -|\Delta|^{\beta} -
\beta \delta z. \sg(\Delta)
|\Delta|^{\beta-1} \right\} 
\nu_{a_-,a_+}^\alpha(dz)\\
=& |\Delta|^\beta \intrs
\left\{|1+(\delta/\Delta) z|^{\beta} -1-
\beta (\delta/\Delta) z \right\} 
\nu_{a_-,a_+}^\alpha(dz) \\
=&\indiq_{\{\delta\Delta>0\}} |\Delta|^\beta \intrs
\left\{\Big|1+|\delta/\Delta| z\Big|^{\beta} -1-
\beta |\delta/\Delta| z \right\} 
\nu_{a_-,a_+}^\alpha(dz) \\
&+\indiq_{\{\delta\Delta<0\}}|\Delta|^\beta \intrs
\left\{\Big|1+|\delta/\Delta| z\Big|^{\beta} -1-
\beta |\delta/\Delta| z \right\} 
\nu_{a_+,a_-}^\alpha(dz) \\
=&|\Delta|^{\beta-\alpha}|\delta|^\alpha
\left[\indiq_{\{\delta\Delta>0\}}\tI^{\alpha,\beta}_{a_-,a_+} 
+\indiq_{\{\delta\Delta<0\}}
\tI^{\alpha,\beta}_{a_+,a_-}   \right].
\end{align*}
We finally have put $x=|\delta/\Delta|z$, 
for which
$\nu_{a_-,a_+}^\alpha(dz)=(|\delta|/|\Delta|)^\alpha \nu_{a_-,a_+}^\alpha(dx)$.

\vip

To study $\tK_{a_-,a_+}^{\alpha,\beta,\eta}(\Delta,\delta)$, we note that
\begin{align*}
&\int_{|\delta z|\leq |\Delta|} 
\min \left\{|\Delta|^{\beta-1} |\delta z|, |\delta z|^\beta\right\}^2 
|z|^{-\alpha-1} dz \\
\leq &
|\Delta|^{2\beta-2}|\delta|^2 \int_{|\delta z|\leq |\Delta|} |z|^{1-\alpha} dz= C 
|\Delta|^{2\beta-\alpha}|\delta|^\alpha.
\end{align*}
This implies that $\tK_{a_-,a_+}^{\alpha,\beta,\eta}(\Delta,\delta) \leq C 
|\Delta|^{2\beta-\alpha}|\delta|^\alpha$ and that we can apply the dominated 
convergence theorem to get 
$\lim_{\eta \to 0} \tK_{a_-,a_+}^{\alpha,\beta,\eta}(\Delta,\delta)=0$.

\vip

Finally, 
\begin{align*}
&\int_{|\delta z|> |\Delta|} 
\min \left\{|\Delta|^{\beta-1} |\delta z|, |\delta z|^\beta\right\}
|z|^{-\alpha-1} dz \\
\leq &
|\delta|^\beta \int_{|\delta z|> |\Delta|} |z|^{\beta-\alpha-1} dz = C 
|\Delta|^{\beta-\alpha}|\delta|^\alpha
\end{align*}
implies that $\tL_{a_-,a_+}^{\alpha,\beta,\eta}(\Delta,\delta) \leq C 
|\Delta|^{\beta-\alpha}|\delta|^\alpha$ and that
$\lim_{\eta \to 0} \tL_{a_-,a_+}^{\alpha,\beta,\eta}(\Delta,\delta)=0$.
\end{proof}

\section{The case with infinite variation}\label{p1}

We now have all the weapons in hand to study
pathwise uniqueness when $\alpha \in (1,2)$.
We first prove that we can apply It\^o's formula with the function
$|x|^\beta$.

\begin{lem}\label{ito1}
Let $\alpha\in (1,2)$, $0\leq a_- \leq a_+$ and $\beta \in (0,1]$.
Assume that $\sigma,b$ have at most linear growth and that for some
constant $\kappa_0\geq 0$, for some $\beta \in (0,1]$,

$\bullet$ for all $x,y\in \rr$,
$\sg(x-y)(b(x)-b(y))\leq \kappa_0 |x-y|$,

$\bullet$ $\sigma$ is H\"older-continuous with index $(\alpha-\beta)/\alpha$.

Consider two solutions $(X_t)_{t\geq 0}$
and $(\tX_t)_{t\geq 0}$ to (\ref{sde}) started at $x$ and $\tx$,
driven by the same 
$(\alpha,a_-,a_+)$-stable process $(Z_t)_{t\geq 0}$ defined by (\ref{stable}). 
Put $\Delta_t=X_t-\tX_t$ and
$\delta_t=\sigma(X_t)-\sigma(\tX_t)$.
Then a.s., for all $t\geq 0$,
\begin{align*}
|\Delta_t|^\beta=&|x-\tx|^\beta+ \beta \intot \indiq_{\{\Delta_s\ne 0\}}
|\Delta_s|^{\beta-1}\sg(\Delta_s) [b(X_s)-b(\tX_s)] ds \nonumber\\
&+ \intot
\indiq_{\{\Delta_s \ne 0\}} |\Delta_s|^{\beta-\alpha}|\delta_s|^\alpha
\left(\indiq_{\{\delta_s\Delta_s>0\}}\tI^{\alpha,\beta}_{a_-,a_+} 
+\indiq_{\{\delta_s\Delta_s<0\}}
\tI^{\alpha,\beta}_{a_+,a_-}  \right)ds + M_t,
\end{align*}
where $\tI^{\alpha,\beta}_{a_-,a_+} $ was defined in (\ref{dftI}) and where
$(M_t)_{t\geq 0}$ is the $L^1$-martingale given by
$$
M_t=\intot \intrs \left[|\Delta_\sm+\delta_\sm z|^{\beta}
-|\Delta_\sm|^{\beta}  \right] \tN(dsdz).
$$
\end{lem}

\begin{proof}
For $\eta>0$, consider $\phi_\eta(x)=(\eta^2+x^2)^{\beta/2}$
as in Lemma \ref{hyperchiant}.
Applying the It\^o formula, see e.g. Jacod-Shiryaev 
\cite[Theorem 4.57 p 56]{js}, 
we get, recalling (\ref{stable}),
\begin{align*}
\phi_\eta(\Delta_t)=&\phi_\eta(x-\tx)+ 
\intot \intrs \left[\phi_\eta(\Delta_\sm+\delta_\sm z) - 
\phi_\eta(\Delta_\sm) \right] \tN(dsdz) \nonumber\ala
&+ \intot\intrs
\left[ \phi_\eta(\Delta_\sm+\delta_\sm z) - \phi_\eta(\Delta_\sm)
- \delta_\sm z  \phi'_\eta(\Delta_\sm) \right] \nu_{a_-,a_+}^\alpha(dz)\; ds\ala
&+ \int_0^t \phi_\eta'(\Delta_\sm)[b(X_s)-b(\tX_s)]ds\\
=:& \phi_\eta(x-\tx) + M^\eta_t + \intot \tJ_{a_-,a_+}^{\alpha,\beta,\eta}
(\Delta_s,\delta_s) ds + \intot A_s^\eta ds,
\end{align*}
where $\tJ_{a_-,a_+}^{\alpha,\beta,\eta}(\Delta,\delta)$
was defined in Lemma \ref{hyperchiant}.
First, we clearly have a.s.
$$
\lim_{\eta\to 0} \phi_\eta(\Delta_t)=
|\Delta_t|^{\beta} \quad \hbox{and} \quad
\lim_{\eta\to 0} \phi_\eta(x-\tx)=|x-\tx|^{\beta}.
$$
Next, we observe that $\tJ_{a_-,a_+}^{\alpha,\beta,\eta}(\Delta_t,\delta_t)=
\tJ_{a_-,a_+}^{\alpha,\beta,\eta}(\Delta_t,\delta_t)\indiq_{\{\Delta_t\ne 0\}}$, 
since 
$\Delta_t=0$ implies that $\delta_t=0$.
Since $\sigma$ is H\"older-continuous with index $(\alpha-\beta)/\alpha$ 
by assumption, 
we deduce that $|\Delta_t|^{\beta-\alpha}|\delta_t|^\alpha$ 
is uniformly bounded.
Thus, using Lemma \ref{hyperchiant} and  
Lebesgue's dominated convergence theorem, we get a.s.
\begin{align*}
&\lim_{\eta\to 0}  \intot 
\tJ_{a_-,a_+}^{\alpha,\beta,\eta}(\Delta_s,\delta_s) ds\\
=& \intot
\indiq_{\{\Delta_s \ne 0\}} |\Delta_s|^{\beta-\alpha}|\delta_s|^\alpha
\left(\indiq_{\{\delta_s\Delta_s>0\}}\tI^{\alpha,\beta}_{a_-,a_+} 
+\indiq_{\{\delta_s\Delta_s<0\}}
\tI^{\alpha,\beta}_{a_+,a_-}  \right)ds.
\end{align*}
Since $\phi_\eta'(x)=\beta x (\eta^2+x^2)^{(\beta-2)/2}$, we may write
$A^\eta_t= A^{\eta,+}_t-A^{\eta,-}_t$, where
\begin{align*}
A^{\eta,+}_t=& \beta |\Delta_t| (\Delta_t^2+\eta^2)^{(\beta-2)/2} 
\left(\sg(\Delta_t)[b(X_t)-b(\tX_t)]\right)_+,\\
A^{\eta,-}_t=& \beta |\Delta_t| (\Delta_t^2+\eta^2)^{(\beta-2)/2} 
\left(\sg(\Delta_t)[b(X_t)-b(\tX_t)]\right)_-.
\end{align*}
First, $\lim_{\eta\to 0} \intot A^{\eta,-}_s ds = 
\beta \intot \indiq_{\{\Delta_s\ne 0\}}
|\Delta_s|^{\beta-1} 
(\sg(\Delta_s)[b(X_s)-b(\tX_s)])_- ds$ by the monotone
convergence theorem. Next, our assumption on $b$ implies that 
$A^{\eta,+}_t \leq \beta \kappa_0 |\Delta_t|^{\beta}$. 
Hence we can apply the dominated convergence theorem to compute 
$\lim_{\eta\to 0} \intot A^{\eta,+}_s ds$ and we finally get a.s.,
$$
\lim_{\eta \to 0} \intot A^\eta_s ds =  \beta \intot \indiq_{\{\Delta_s\ne 0\}}
|\Delta_s|^{\beta-1} 
\sg(\Delta_s)[b(X_s)-b(\tX_s)] ds
$$
as desired. It only remains to prove that $M^\eta_t$ tends to $M_t$ in $L^1$.
We write $M_t=M_t^1+M_t^2$ and $M_t^\eta=M_t^{\eta,1}+M_t^{\eta,2}$, where
\begin{align*}
M^1_t=&\intot\intrs \indiq_{\{|\delta_\sm z|\leq |\Delta_\sm|\}}
\left[|\Delta_\sm+\delta_\sm z|^{\beta}-|\Delta_\sm|^{\beta}  \right] \tN(dsdz),\\
M^2_t=&\intot\intrs \indiq_{\{|\delta_\sm z|> |\Delta_\sm|\}}
\left[|\Delta_\sm+\delta_\sm z|^{\beta}-|\Delta_\sm|^{\beta}  \right] \tN(dsdz),\\
M^{1,\eta}_t=&\intot\intrs \indiq_{\{|\delta_\sm z|\leq |\Delta_\sm|\}}
\left[\phi_\eta(\Delta_\sm+\delta_\sm z) - \phi_\eta(\Delta_\sm)\right] \tN(dsdz),\\
M^{2,\eta}_t=&\intot\intrs \indiq_{\{|\delta_\sm z|> |\Delta_\sm|\}}
\left[\phi_\eta(\Delta_\sm+\delta_\sm z) - \phi_\eta(\Delta_\sm)\right] \tN(dsdz).
\end{align*}
Using Lemma \ref{hyperchiant} and
Lebesgue's dominated convergence theorem, there holds
$$
\lim_{\eta\to0}\E\left[|M^1_t-M^{1,\eta}_t|^2\right] = \lim_{\eta\to0}
\intot \E\left[\tK_{a_-,a_+}^{\alpha,\beta,\eta}(\Delta_s,\delta_s)  \right] ds=0,
$$
since $\tK_{a_-,a_+}^{\alpha,\beta,\eta}(\Delta_s,\delta_s)\leq C 
|\Delta_t|^{2\beta-\alpha}|\delta_t|^\alpha\leq C |\Delta_t|^\beta$
and since $\E[\sup_{[0,t]} |\Delta_s|^\beta]<\infty$ by 
Proposition \ref{weak}-(ii). Similarly, writing $\tN(dsdz)=N(dsdz)-
\nu_{a_-,a_+}^\alpha(dz)\; dz$, 
$$
\lim_{\eta\to0}\E\left[|M^2_t-M^{2,\eta}_t|\right] \leq 2\lim_{\eta\to0} 
\intot \E\left[\tL_{a_-,a_+}^{\alpha,\beta,\eta}(\Delta_s,\delta_s)  \right] ds=0,
$$
because $\tL_{a_-,a_+}^{\alpha,\beta,\eta}(\Delta_s,\delta_s)\leq C 
|\Delta_t|^{\beta-\alpha}|\delta_t|^\alpha\leq C$. This ends the proof.
\end{proof}

\begin{preuve} {\it of Theorem \ref{mr1}.} 
We thus consider $\alpha \in (1,2)$, $0\leq a_-\leq a_+$ and 
two solutions $(X_t)_{t\geq 0}$ and $(\tX_t)_{t\geq 0}$ to (\ref{sde}).
We put $\Delta_t=X_t-\tX_t$
and $\delta_t=\sigma(X_t)-\sigma(\tX_t)$. We set
$\beta=\beta(\alpha,a_-/a_+)\in [\alpha-1,1]$ 
as in Lemma \ref{acos0} and we use Lemma \ref{ito1}. With our choice
for $\beta$, there holds $\tI^{\alpha,\beta}_{a_-,a_+}=0$ by Lemma \ref{basic}. 
We thus find
\begin{align*}
|\Delta_t|^\beta=&|x-\tx|^\beta+ \beta \intot \indiq_{\{\Delta_s\ne 0\}}
|\Delta_s|^{\beta-1}\sg(\Delta_s) [b(X_s)-b(\tX_s)] ds \nonumber\\
&+ C\indiq_{\{a_-\ne a_+\}} \intot \indiq_{\{\Delta_s \delta_s<0\}}
|\Delta_s|^{\beta-\alpha}
|\delta_s|^\alpha ds +M_t,
\end{align*}
where $C=I^{\alpha,\beta}_{a_+,a_-}$ and where $(M_t)_{t\geq 0}$ is a 
$L^1$-martingale.

\vip

{\it Step 1.} We prove point (i). Due to our assumption on $b$,
$\sg(\Delta_s)[b(X_s)-b(\tX_s)] \leq \kappa_0|\Delta_s|$.
If $a_-=a_+$, we thus get,
taking expectations, $\E[|\Delta_t|^\beta] \leq |x-\tx|^\beta
+ \beta \kappa_0 \intot \E[|\Delta_s|^\beta] ds$ and we conclude with 
the Gronwall Lemma. If $a_-<a_+$, our assumption on $\sigma$
guarantees us that if
$\delta_s \Delta_s<0$, then $|\delta_s| \leq \kappa_1 |\Delta_s| /(a_+-a_-)$,
whence $|\Delta_s|^{\beta-\alpha}
|\delta_s|^\alpha\indiq_{\{\delta_s\Delta_s<0\}} \leq C |\Delta_s|^\beta$.
Hence taking expectations, we get $\E[|\Delta_t|^\beta] \leq |x-\tx|^\beta
+ C \intot \E[|\Delta_s|^\beta] ds$:
we also conclude with the Gronwall lemma.

\vip

{\it Step 2.} We check point (ii). Assuming that $(a_+-a_-)\sigma$ 
is non-decreasing,  we deduce that either $a_-=a_+$
or for all $s\geq 0$, a.s., $\delta_s\Delta_s\geq 0$. If furthermore
$b$ is constant, we thus get 
$|\Delta_t|^\beta=|x-\tx|^\beta +M_t$, whence 
$\E[|\Delta_t|^\beta]=|x-\tx|^\beta$.
\end{preuve}

We now study the large time behavior of solutions when $a_-=a_+$.

\begin{preuve} {\it of Proposition \ref{inv}.}
We thus assume that $a_-=a_+>0$, that $\alpha \in (1,2)$, that $b$ 
is non-increasing and continuous
and that $\sigma$ is H\"older-continuous with index $1/\alpha$.

\vip

{\it Step 1.} Consider any pair of solutions $(X_t)_{t\geq 0}$,
$(\tX_t)_{t\geq 0}$ to (\ref{sde}) driven by the same stable process
and set, as usual, $\Delta_t=X_t-\tX_t$, $\delta_t=\sigma(X_t)-\sigma(\tX_t)$.
Lemma \ref{ito1} with $\beta=\beta(\alpha,1)=\alpha-1$ implies,
since 
$I_{a_-,a_+}^{\alpha,\beta}=I_{a_+,a_-}^{\alpha,\beta}=0$
by Lemma \ref{basic}, that
\begin{align*}
|\Delta_t|^{\alpha-1}=&|x-\tx|^{\alpha-1}+M_t +  
(\alpha-1)\int_0^t\sg(\Delta_s)|\Delta_s|^{\alpha-2}[b(X_s)-b(\tX_s)]ds,
\end{align*}
where 
$M_t=\intot \intrs \left(|\Delta_\sm+\delta_\sm z|^{\alpha-1}
-|\Delta_\sm|^{\alpha-1}  \right)\tN(dsdz)$.
Using that $b$ is non-increasing, we deduce that
\begin{align}\label{no1}
|\Delta_t|^{\alpha-1}+  
(\alpha-1)\int_0^t|\Delta_s|^{\alpha-2}|b(X_s)-b(\tX_s)|ds
=&|x-\tx|^{\alpha-1}+M_t=:U_t.
\end{align}
Consequently, $U_t$ is a non-negative martingale.
Thus it a.s. converges as $t\to\infty$, as well as its bracket:
\begin{align*}
\int_0^\infty \intrs \left[|\Delta_s+\delta_s z|^{\alpha-1}
-|\Delta_s|^{\alpha-1} \right]^2 \nu_{a_-,a_+}^\alpha(dz)\; ds<\infty.
\end{align*}
But if $\Delta_s \ne 0$, since $a_-=a_+$, setting $c=a_+ \intrs
[|1+x|^{\alpha-1}-1]^2 |x|^{-\alpha-1}dx >0$,
\begin{align*}
&\intrs \left[|\Delta_s+\delta_s z|^{\alpha-1}
-|\Delta_s|^{\alpha-1} \right]^2 \nu_{a_-,a_+}^\alpha(dz)\\
=&a_+|\Delta_s|^{2\alpha-2}
\intrs \left[|1+\delta_s z/\Delta_s|^{\alpha-1}
-1 \right]^2|z|^{-\alpha-1}dz \\
=&c |\Delta_s|^{2\alpha-2} (|\delta_s/\Delta_s|)^\alpha=
c  |\Delta_s|^{\alpha-2}|\delta_s|^\alpha,
\end{align*}
whence 
\begin{align}\label{esti1}
\int_0^\infty \indiq_{\{\Delta_s\ne 0\}} |\Delta_s|^{\alpha-2}|\delta_s|^\alpha 
ds <\infty \quad\hbox{a.s.}
\end{align}
We also have $\sup_{[0,\infty)} U_t <\infty$, whence, due to (\ref{no1}),
\begin{align}\label{esti2}
\int_0^\infty|\Delta_s|^{\alpha-2}|b(X_s)-b(\tX_s)|ds <\infty \quad\hbox{a.s.}
\end{align}
Finally, Doob's $L^1$ inequality (see e.g. Revuz-Yor 
\cite[Theorem 1.7 p 54]{ry}) implies, since $U_t$
is a non-negative $L^1$ c\`adl\`ag martingale, that for any $a>0$, 
$$
\Pr[\sup_{[0,\infty)}U_t \geq a]\leq a^{-1}\sup_{[0,\infty)}\E[U_t]=
a^{-1}|x-\tx|^{\alpha-1}.
$$ 
But $\sup_{[0,\infty)}|\Delta_t|^{\alpha-1}\leq
\sup_{[0,\infty)}U_t$ by (\ref{no1}). Hence for any $\beta\in (0,\alpha-1)$, for any
$c>0$,
\begin{align*}
\E\left[\sup_{[0,\infty)}|\Delta_t|^\beta \right]=&
\int_0^\infty \Pr\left[\sup_{[0,\infty)}|\Delta_t|^{\alpha-1} 
\geq a^{(\alpha-1)/\beta} \right] da\\
\leq & c + \int_c^\infty a^{-(\alpha-1)/\beta}|x-\tx|^{\alpha-1} da \\
=& c + \frac{\beta}{\alpha-1-\beta}|x-\tx|^{\alpha-1} c^{1-(\alpha-1)/\beta}.
\end{align*}
Choose $c=|x-\tx|^{\beta}$: for some constant $C$ 
depending only on $\alpha,\beta$,
\begin{align}\label{esti3}
\E\left[\sup_{[0,\infty)}|\Delta_t|^\beta\right]\leq C |x-\tx|^{\beta}.
\end{align}

\vip

{\it Step 2.} We now prove point (i).
Consider two invariant distributions $Q$ and $\tQ$ for
(\ref{sde}). Let $X_0\sim Q$ and $\tX_0\sim \tQ$ 
be two random variables independent of $(Z_t)_{t\geq 0}$.
Consider the associated solutions $(X_t)_{t\geq 0}$ and $(\tX_t)_{t\geq 0}$
to (\ref{sde}) starting from $X_0$ and $\tX_0$ 
(pathwise existence holds for (\ref{sde}): we have checked pathwise 
uniqueness in Theorem \ref{mr1} 
and weak existence in Proposition \ref{weak}).
Then we have $X_t\sim Q$ and $\tX_t\sim \tQ$ for all $t\geq 0$.
>From (\ref{esti1}) and (\ref{esti2}), we have a.s.
$$
\int_0^\infty \Gamma(X_t,\tX_t) dt <\infty,
$$
where 
\begin{align*}
\Gamma(x,y):=&(1+|x-y|)^{\alpha-2} 
[|b(x)-b(y)|+|\sigma(x)-\sigma(y)|^\alpha] \\
\leq&
\indiq_{\{x \ne y\}}|x-y|^{\alpha-2} 
[|b(x)-b(y)|+|\sigma(x)-\sigma(y)|^\alpha].
\end{align*}
We easily deduce, see e.g. \cite[Lemma 10]{fp}, that there is a deterministic
sequence $(t_n)_{n\geq 1}$ increasing to infinity such that 
$\Gamma(X_{t_n},\tX_{t_n})$ goes to $0$ in probability.
Since $(\sigma,b)$ is injective by assumption, we have $\Gamma(x,y) >0$ 
for all $x \ne y$. Furthermore, $\Gamma$ is continuous and
$X_{t_n}\sim Q$ and $\tX_{t_n}\sim \tQ$ for all $n\geq 1$. We thus
infer from \cite[Lemma 11]{fp} that $Q=\tQ$.

\vip

{\it Step 3.} 
We next prove point (ii). Consider two solutions $(X_t)_{t\geq 0}$
and $(\tX_t)_{t\geq 0}$ to (\ref{sde}), issued from $x$ and $\tx$.
Using our assumptions and (\ref{esti1})-(\ref{esti2}), we get
$$
\int_0^\infty \rho(|X_t-\tX_t|)dt <\infty.
$$
Hence, see e.g. \cite[Lemma 10]{fp}, there is a deterministic
sequence $(t_n)_{n\geq 1}$ increasing to infinity such that
$\rho(|X_{t_n}-\tX_{t_n}|)$ goes to $0$ in probability. Since
$\rho$ is strictly increasing and vanishes only at $0$, we deduce that
$|X_{t_n}-\tX_{t_n}|$ goes to $0$ in probability. We thus infer 
from (\ref{esti3}), choosing e.g. $\beta=(\alpha-1)/2$, that
$$
\E\left[\left. \sup_{[t_n,\infty)}|X_t-\tX_t|^\beta \right\vert \cF_{t_n}
\right] \leq C |X_{t_n}-\tX_{t_n}|^\beta.
$$
We used that conditionally on $\cF_{t_n}$, $(X_{t_n+t})_{t\geq 0}$
and $(\tX_{t_n+t})_{t\geq 0}$ solve (\ref{sde}). We easily deduce that
$\sup_{[t_n,\infty)}|X_t-\tX_t|$ tends to $0$ in probability. Since finally
$s\mapsto \sup_{[s,\infty)}|X_t-\tX_t|$ is non-increasing, it a.s.
admits a limit as $s\to \infty$ and this limit can
only be $0$.
\end{preuve}

\section{The case with finite variation}\label{p2}

We now study the case where $\alpha\in (0,1)$.
Here again, we first prove that we can apply It\^o's formula with the function
$|x|^\beta$.

\begin{lem}\label{ito2}
Let $\alpha\in (0,1)$, $0\leq a_- \leq a_+$ and $\beta \in (0,\alpha)$.
Assume that $\sigma,b$ have at most linear growth and that for some
constant $\kappa_0\geq 0$, for some $\beta \in (0,1]$,

$\bullet$ for all $x,y\in \rr$,
$\sg(x-y)(b(x)-b(y))\leq \kappa_0 |x-y|$,

$\bullet$ $\sigma$ is H\"older-continuous with index $\theta$ for some
$\theta\in [\alpha-\beta,\alpha]$.

Consider two solutions $(Y_t)_{t\geq 0}$
and $(\tY_t)_{t\geq 0}$ to (\ref{sde}) started at $x$ and $\tx$,
driven by the same Poisson measure $M$.
Put $\Delta_t=Y_t-\tY_t$. Then for all $t\geq 0$, recall (\ref{dfI}),
\begin{align*}
\E[|\Delta_t|^\beta]\leq &|x-\tx|^\beta + \beta \kappa_0 \intot \E\left[
|\Delta_s|^{\beta}\right] ds\\
&+ \intot \E\left[(\gamma(Y_s)-\gamma(\tY_s))_+|\Delta_s|^{\beta-\alpha}
[\indiq_{\{\Delta_s>0\}}
I^{\alpha,\beta}_{a_+,a_-} + \indiq_{\{\Delta_s<0\}}I^{\alpha,\beta}_{a_-,a_+}
] \right] ds\\
&+ \intot \E\left[(\gamma(\tY_s)-\gamma(Y_s))_+|\Delta_s|^{\beta-\alpha}
[\indiq_{\{\Delta_s>0\}}
I^{\alpha,\beta}_{a_-,a_+} + \indiq_{\{\Delta_s<0\}}I^{\alpha,\beta}_{a_+,a_-}
] \right] ds,
\end{align*}
with an equality and $\kappa_0=0$ if $b$ is constant.
\end{lem}

\begin{proof}
We define, for 
$y,\ty \in \rr$ and $u \in \rr_*$, 
$$
\Gamma(y,\ty,u)=\indiq_{\{0<u<\gamma(y)\}}-\indiq_{\{\gamma(y)<u<0\}}
-\indiq_{\{0<u<\gamma(\ty)\}}+\indiq_{\{\gamma(\ty)<u<0\}}.
$$
Let also $\phi_\eta(x)=(\eta^2+x^2)^{\beta/2}$.
Applying the It\^o formula for jump processes, see e.g.
\cite[Theorem 4.57 p 56]{js}, we get
\begin{align*}
\phi_\eta(\Delta_t)=& \phi_\eta(x-\tx) 
+\intot \phi_\eta'(\Delta_s)(b(Y_s)-b(\tY_s))ds\\
&+\intot \intrs \intrs
\left[\phi_\eta\left(\Delta_\sm+z\Gamma(Y_\sm,\tY_\sm,u)\right) 
-\phi_\eta(\Delta_\sm)\right]
M(dsdzdu).
\end{align*}
First, since $|\phi_\eta'(x)|\leq \beta |x|^{\beta-1}$ and since
$\sg(\phi_\eta'(x))=\sg(x)$, we deduce from our assumption on $b$
that for any $\eta>0$, any $y,\ty$, 
\begin{align*}
\phi_\eta'(y-\ty) (b(y)-b(\ty))
=&|\phi_\eta'(y-\ty)| \sg(y-\ty)(b(y)-b(\ty))\\
\leq& |\phi_\eta'(y-\ty)| \kappa_0|y-\ty|\\
\leq &
\beta\kappa_0 |y-\ty|^\beta. 
\end{align*}
We deduce that
\begin{align}\label{abcd}
\phi_\eta(\Delta_t)\leq & \phi_\eta(x-\tx) + \beta \kappa_0 \intot
|\Delta_s|^\beta ds \nonumber \\
&+ \intot \intrs\intrs 
\left[\phi_\eta\left(\Delta_\sm+z\Gamma(Y_\sm,\tY_\sm,u)\right) 
-\phi_\eta(\Delta_\sm)\right]
M(dsdzdu),
\end{align}
with of course an equality and $\kappa_0=0$ if $b$ is constant.
Observe now that for any $y,\ty \in \rr$, any $u\in \rr_*$,
$$
\Gamma(y,\ty,u)=\indiq_{\{\gamma(\ty)<u<\gamma(y)\}} 
-\indiq_{\{\gamma(y)<u<\gamma(\ty)\}} .
$$
Hence integrating in $u$ and recalling Lemma \ref{hyperchiant2},
\begin{align*}
&\intrs \intrs \left|\phi_\eta\left(\Delta_\sm+z\Gamma(Y_\sm,\tY_\sm,u)\right) 
-\phi_\eta(\Delta_\sm)\right| \nu_{a_-,a_+}^\alpha(dz) \; du\ala
=&(\gamma(Y_s)-\gamma(\tY_s))_+ \intrs
\left|\phi_\eta\left(\Delta_s+z\right)-\phi_\eta(\Delta_s)\right| 
\nu_{a_-,a_+}^\alpha(dz)\ala
&+(\gamma(\tY_s)-\gamma(Y_s))_+ \intrs
\left|\phi_\eta\left(\Delta_s-z\right)-\phi_\eta(\Delta_s)\right| 
\nu_{a_-,a_+}^\alpha(dz)\ala
=& (\gamma(Y_s)-\gamma(\tY_s))_+ K_{a_+,a_-}^{\alpha,\beta,\eta}(\Delta_s,\delta_s)
+(\gamma(\tY_s)-\gamma(Y_s))_+K_{a_-,a_+}^{\alpha,\beta,\eta}(\Delta_s,\delta_s)\ala
\leq & C |\gamma(Y_s)-\gamma(\tY_s)|. |\Delta_s|^{\beta-\alpha}.
\end{align*}
Since $\gamma$ is H\"older-continuous with index $\theta$, this is bounded
by $C |\Delta_s|^{\theta-\alpha+\beta}$. 
Using Proposition \ref{weak}-(ii) and that $\theta-\alpha+\beta \in [0,\alpha)$,
we can thus take expectations in 
(\ref{abcd}):
\begin{align}\label{abcde}
\E[\phi_\eta(\Delta_t)]\leq& \phi_\eta(|x-\tx|) + \beta \kappa_0 \intot
\E\left[|\Delta_s|^\beta\right] ds + \intot \E \left[B^\eta_s\right]ds,
\end{align}
(with an equality and $\kappa_0=0$ if $b$ is constant), where
\begin{align*}
B^\eta_s =\intrs\intrs
\left[\phi_\eta\left(\Delta_\sm+z\Gamma(Y_\sm,\tY_\sm,u)\right) 
-\phi_\eta(\Delta_\sm)\right] du \; \nu_{a_-,a_+}^\alpha(dz).
\end{align*}
First, we obviously have $\lim_{\eta\to 0} \E[\phi_\eta(\Delta_t)]
=\E[|\Delta_t|^\beta]$ and $\lim_{\eta\to 0} \phi_\eta(x-\tx)=|x-\tx|^\beta$.
Next, integrating in $u$ as previously and 
recalling Lemma \ref{hyperchiant2}, we obtain
\begin{align*}
B^\eta_s =& (\gamma(Y_s)-\gamma(\tY_s))_+ \intrs
\left[\phi_\eta\left(\Delta_s+z\right)-\phi_\eta(\Delta_s)\right]
\nu_{a_-,a_+}^\alpha(dz)\ala
&+(\gamma(\tY_s)-\gamma(Y_s))_+ \intrs
\left[\phi_\eta\left(\Delta_s-z\right)-\phi_\eta(\Delta_s)\right]
\nu_{a_-,a_+}^\alpha(dz)\\
=&(\gamma(Y_s)-\gamma(\tY_s))_+ J^{\alpha,\beta,\eta}_{a_+,a_-}(\Delta_s) 
+(\gamma(\tY_s)-\gamma(Y_s))_+ J^{\alpha,\beta,\eta}_{a_-,a_+}(\Delta_s).
\end{align*}
For the first integral, we have used the substitution
$x=-z$, so that $\nu_{a_-,a_+}^\alpha(dz)=\nu_{a_+,a_-}^\alpha(dx) $. 
Using finally  Lemma \ref{hyperchiant2}, 
that $\gamma$ is H\"older continuous with index $\theta\in
[\alpha-\beta,\alpha]$, Proposition
\ref{weak}-(ii) (since $\theta+\beta-\alpha \in [0,\alpha)$)
and the Lebesgue dominated convergence theorem, we
deduce that 
\begin{align*}
&\lim_{\eta\to 0} \intot \E \left[B^\eta_s\right]ds\\
=& \intot \E\left[(\gamma(Y_s)-\gamma(\tY_s))_+|\Delta_s|^{\beta-\alpha}
[\indiq_{\{\Delta_s>0\}}
I^{\alpha,\beta}_{a_+,a_-} + \indiq_{\{\Delta_s<0\}}I^{\alpha,\beta}_{a_-,a_+}
] \right] ds \\
&+\intot \E\left[(\gamma(\tY_s)-\gamma(Y_s))_+|\Delta_s|^{\beta-\alpha}
[\indiq_{\{\Delta_s>0\}}
I^{\alpha,\beta}_{a_-,a_+} + \indiq_{\{\Delta_s<0\}}I^{\alpha,\beta}_{a_+,a_-}
] \right] ds
\end{align*}
as desired.
\end{proof}

We can now give the 

\begin{preuve} {\it of Theorem \ref{mr2}.}
We consider $\alpha \in (0,1)$, $a_-\leq a_+$ and 
two solutions $(Y_t)_{t\geq 0}$ and $(\tY_t)_{t\geq 0}$ to (\ref{sde2}),
issued from $x$ and $\tx$. We also fix $\beta \in (0,\alpha)$.
We put $\Delta_t=Y_t-\tY_t$. Applying Lemma \ref{ito2} and recalling
that $\gamma$ is H\"older continuous with index $\alpha$, a rough
upperbound using only that 
$|I^{\alpha,\beta}_{a_-,a_+}|+|I^{\alpha,\beta}_{a_+,a_-}|<\infty$
yields that $\E[|\Delta_t|^\beta] \leq |x-\tx|^\beta + C \intot 
\E[|\Delta_s|^\beta] ds$ and we conclude with the Gronwall Lemma.
\end{preuve}

We conclude this section with the

\begin{preuve} {\it of Theorem \ref{mrsuper}.}
Let us thus assume that $\alpha\in (1/2,1)$, that $a_-<a_+$
with $a_-/a_+ < -\cos (\pi\alpha)$ and let us set $\beta=\beta(\alpha,a_-/a_+)
\in (0,\alpha)$.
Consider 
two solutions $(Y_t)_{t\geq 0}$ and $(\tY_t)_{t\geq 0}$ to (\ref{sde2}),
issued from $x$ and $\tx$ and put $\Delta_t=Y_t-\tY_t$.
Applying Lemma \ref{ito2} ($\gamma$ is H\"older-continuous with index
$\alpha-\beta$ by assumption) and recalling that $I_{a_-,a_+}^{\alpha,\beta}=0$
due to Lemma \ref{basic}, we get
\begin{align*}
\E[|\Delta_t|^\beta]\leq &|x-\tx|^\beta + \beta \kappa_0 \intot \E\left[
|\Delta_s|^{\beta}\right] ds + \intot 
\E\left[B^{\eta,1}_s+B^{\eta,2}_s \right] ds,
\end{align*}
(with an equality and $\kappa_0=0$ if $b$ is constant), where
\begin{align*}
B^{\eta,1}_s=& I^{\alpha,\beta}_{a_+,a_-}
(\gamma(Y_s)-\gamma(\tY_s))_+|\Delta_s|^{\beta-\alpha}\indiq_{\{\Delta_s>0\}},
\\
B^{\eta,2}_s=& I^{\alpha,\beta}_{a_+,a_-}(\gamma(\tY_s)-\gamma(Y_s))_+
|\Delta_s|^{\beta-\alpha}\indiq_{\{\Delta_s<0\}}.
\end{align*}

{\it Step 1.} We now prove point (i). Our assumption on $\gamma$
guarantees that if $\Delta_s>0$, then $(\gamma(Y_s)-\gamma(\tY_s))_+
\leq \kappa_1 |\Delta_s|^\alpha$, whence $B^{\eta,1}_s \leq C |\Delta_s|^\beta$.
Similarly,  $B^{\eta,2}_s \leq C |\Delta_s|^\beta$. We thus find
$\E[|\Delta_t|^\beta]\leq |x-\tx|^\beta + C\intot \E\left[
|\Delta_s|^{\beta}\right] ds$ and we conclude with the Gronwall Lemma.

\vip

{\it Step 2.} We now check point (ii), assuming that $b$ is constant
and that $\gamma$ is non-increasing. Then $\Delta_s>0$ implies
$\gamma(Y_s)-\gamma(\tY_s)\leq 0$, whence $B^{\eta,1}_s=0$.
Similarly,  $B^{\eta,2}_s=0$ and we obtain 
$\E[|\Delta_t|^\beta]= |x-\tx|^\beta$ as desired.
\end{preuve}

\section{Weak existence and equivalence of the two equations}\label{ch}

We start this section with the equivalence in law
between (\ref{sde}) and (\ref{sde2}).

\begin{preuve} {\it of Lemma \ref{eq}}. We fix $\alpha \in (0,1)$, 
$a_-,a_+ \in [0,\infty)$ and we start with point (i).
We thus consider a solution 
$(Y_t)_{t\geq 0}$
to (\ref{sde2}) driven by a Poisson measure $M$ with intensity measure
$ds \; \nu_{a_-,a_+}^\alpha(dz) \; du$. 
Recall that $\gamma(x)=\sg(\sigma(x)).|\sigma(x)|^\alpha$.
Set
\begin{align*}
Z_t=&\intot \intrs \intrs \Big\{ \indiq_{\{\sigma(Y_\sm)\ne 0\}} 
\frac{z}{\sigma(Y_\sm)}[\indiq_{\{0<u< \gamma(Y_\sm)\}} 
-\indiq_{\{\gamma(Y_\sm)<u<0\}}]\\
&\hskip5cm +\indiq_{\{\sigma(Y_\sm)= 0\}} z \indiq_{\{0<u<1\}}
\Big\}M(dsdzdu).
\end{align*}
Then we obviously have 
\begin{align*}
\intot \sigma(Y_\sm)dZ_s=
\intot\intrs\intrs z [\indiq_{\{0<u< \gamma(Y_\sm)\}} -\indiq_{\{\gamma(Y_\sm)<u<0\}}]
M(dsdzdu).
\end{align*}
It only remains to prove that $(Z_t)_{t\geq 0}$ is a 
$(\alpha,a_-,a_+)$-stable process. 
But $(Z_t)_{t\geq 0}$ is a pure jump
process without drift, so that we only need to check that,
for $J=\{s\in [0,\infty), \Delta Z_s\ne 0\}$, 
$\sum_{s>0} \indiq_{\{s\in J\}}\delta_{(s,\Delta Z_s)}$ is a 
Poisson measure
on $[0,\infty)\times \rr_*$ with intensity measure
$ds \; \nu^\alpha_{a_-,a_+}(dz)$. Denote by $q(dsdz)$ its compensator.
It is enough (see Jacod-Shiryaev \cite[Theorem 4.8 p 104]{js}) 
to show that $q(dsdz)=\nu^\alpha_{a_-,a_+}(dz)\;ds$. 
By Definition of
$(Z_t)_{t\geq 0}$, we clearly have (recall that $\sg(\sigma(x))=\sg(\gamma(x))$)
\begin{align*}
&\intot\intrs \phi(s,z)q(dsdz)\\
=& \intot
\intrs \intrs  \indiq_{\{\sigma(Y_s) >0 \}}
\phi(s,z/\sigma(Y_s))\indiq_{\{0<u<\gamma(Y_s)\}}  du \; 
\nu^\alpha_{a_-,a_+}(dz)\; ds\\
&+ \intot
\intrs \intrs  \indiq_{\{\sigma(Y_s) <0 \}}
\phi(s,-z/\sigma(Y_s))\indiq_{\{\gamma(Y_s)<u<0\}}  du \; 
\nu^\alpha_{a_-,a_+}(dz)\; ds\\
&+\intot \intrs \intrs
\indiq_{\{\sigma(Y_s) = 0 \}}\phi(s,z) \indiq_{\{0<u<1\}}  du \; 
\nu^\alpha_{a_-,a_+}(dz) \; ds.
\end{align*}
Integrating in $u$, we deduce that
\begin{align*}
\intot\intrs \phi(s,z)q(dsdz)=& \intot
\intrs \indiq_{\{\sigma(Y_s) >0 \}}
\phi(s,z/\sigma(Y_s))\gamma(Y_s)\nu^\alpha_{a_-,a_+}(dz) \; ds\\
&+ \intot
\intrs \indiq_{\{\sigma(Y_s) <0 \}}
\phi(s,-z/\sigma(Y_s))|\gamma(Y_s)|  \nu^\alpha_{a_-,a_+}(dz) \; ds\\
&+ \intot \intrs
\indiq_{\{\sigma(Y_s) = 0 \}}\phi(s,z) \nu^\alpha_{a_-,a_+}(dz)\; ds.
\end{align*}
We perform the substitution $x=z/|\sigma(Y_s)|$ in the two first integrals,
which yields that 
$\nu_{a_-,a_+}^\alpha(dz)=|\sigma(Y_s)|^{-\alpha}\nu_{a_-,a_+}^\alpha(dx)$.
Recalling that 
$|\sigma(Y_s)|^{-\alpha} |\gamma(Y_s)|=1$, 
we conclude that $\intot\intrs \phi(s,z)q(dsdz)=\intot 
\intrs \phi(s,z) \nu^\alpha_{a_-,a_+}(dz)\;ds$ as desired.

\vip

We now check point (ii). Let thus 
$(X_t)_{t\geq 0}$  
solve (\ref{sde}) with some $(\alpha,a_-,a_+)$-stable process $(Z_t)_{t\geq 0}$
Put $N=\sum_{s>0} \indiq_{\{s\in J\}}\delta_{(s,\Delta Z_s)}$,
which is a Poisson measure
on $[0,\infty)\times \rr_*$ with intensity measure
$ds \; \nu^\alpha_{a_-,a_+}(dz)$. On an enlarged probability space, 
we consider a Poisson measure $O(dsdzdu)$  on 
$[0,\infty)\times \rr_*\times \rr_*$  with intensity measure
$ds \; \nu^\alpha_{a_-,a_+}(dz) \; du$ such that $N(dsdz)=O(dsdz\times[0,1])$.
We finally introduce the random point measure $M(dsdzdu)$ on 
$[0,\infty)\times \rr_*\times \rr_*$ defined by
\begin{align*}
&\intot \intrs \intrs \varphi(s,z,u) M(dsdudz)\\
=&\intot \intrs \intrs \indiq_{\{\sigma(X_\sm)\ne 0\}}
\varphi(s,z|\sigma(X_\sm)|,u\gamma(X_\sm)) O(dsdudz)\\
&+ \intot \intrs \intrs \indiq_{\{\sigma(X_\sm)= 0\}}
\varphi(s,z,u) O(dsdudz).
\end{align*}
for all $\varphi$ smooth enough. Then we have
\begin{align*}
\intot \sigma(X_\sm)&dZ_s = \intot \intrs z \sigma(X_\sm) N(dsdz)\\
=& \intot \intrs \intrs z \sigma(X_\sm)\indiq_{\{u\in [0,1]\}} O(dsdzdu)\\
=& \intot \intrs \intrs  z .\sg(\sigma(X_\sm))\indiq_{\{\sigma(X_\sm)\ne 0\}}
\indiq_{\{u/\gamma(X_\sm)\in [0,1]\}} M(dsdzdu)\\
=&\intot \intrs z\left[\indiq_{\{0<u<\gamma(X_\sm)\}} 
- \indiq_{\{\gamma(X_\sm)<u<0\}} \right]M(dsdudz).
\end{align*}
We finally used that $\sg(\sigma(x))=\sg(\gamma(x))$
and that $\sigma(x)=0$ implies $\gamma(x)=0$. It thus only remains
to check that $M$ is a Poisson measure  with intensity 
measure 
$ds \; \nu^\alpha_{a_-,a_+}(dz) \; du$. Let us call $p$ the compensator
of $M$ and observe that
\begin{align*}
&\intot\intrs\intrs \phi(s,z,u) p(dsdzdu)\\
=&\intot\intrs\intrs \indiq_{\{\sigma(X_\sm)\ne 0\}}
\phi(s,z|\sigma(X_\sm)|,u\gamma(X_\sm)) du \; \nu^\alpha_{a_-,a_+}(dz) \;  ds \\
&+ \intot\intrs\intrs \indiq_{\{\sigma(X_\sm)= 0\}}
\phi(s,z,u) du \; \nu^\alpha_{a_-,a_+}(dz) \;  ds.
\end{align*}
Performing the substitution $v=z|\sigma(X_\sm)|$, $w=u\gamma(X_\sm)$
and recalling that $|\sigma(X_\sm)|^{-\alpha}|\gamma(X_\sm)|=1$, 
we easily conclude
that $p(dsdzdu)=ds \; \nu^\alpha_{a_-,a_+}(dz)\; du$, which ends the proof.
\end{preuve}

We finally end this paper with weak existence and moment estimates
for (\ref{sde}).

\begin{preuve} {\it of Proposition \ref{weak}.} 
Let us for example treat the case where $\alpha \in (1,2)$.
Consider the equation
\begin{align}\label{edssgs}
Y_t=x+\intot \int_{-1}^1 \sigma(Y_\sm)z \tN(dsdz) + \intot c(Y_s)ds,
\end{align}
where $c(x)=b(x)- \sigma(x) \int_{|z|\geq 1} z \nu_{a_-,a_+}^\alpha(dz)$.
If $b$ and $\sigma$ have at most linear growth, one immediately
checks, using that $\int_{-1}^1 z^2 \nu_{a_-,a_+}^\alpha(dz)<\infty$, 
that for any $T>0$, for some constant $C_T$ not depending on $x$, any solution
to (\ref{edssgs}) satisfies 
\begin{align}\label{fme}
\E\left[\sup_{[0,T]} Y_t^2\right] \leq C_T (1+x^2).
\end{align}
If furthermore $b$ and $\sigma$ are continuous, we can apply
Theorem 175 of Situ \cite{s}
and thus weak existence holds for (\ref{edssgs}). 
Rewrite now (\ref{sde}) as 
\begin{align}\label{sdeprime}
X_t=x+\intot \int_{-1}^1 \sigma(X_\sm)z\tN(dsdz) + \intot c(X_s)ds 
+ \intot \int_{|z|>1} \sigma(X_\sm)zN(dsdz).
\end{align}
Observe that the last integral generates jumps at some discrete instants:
rewrite the restriction of $N$ to $[0,\infty)\times\{|z|\geq1\}$
as $\sum_{n\geq 1} \delta_{(T_n,Z_n)}$, where $0<T_1<T_2<\dots$ are the jump
instants of a Poisson process with parameter 
$\lambda=\int_{|z|\geq 1} \nu_{a_-,a_+}^\alpha(dz)$ and where the random variables
$(Z_n)_{n\geq 1}$ are i.i.d. with law $\lambda^{-1}  \nu_{a_-,a_+}^\alpha(dz)$.
Hence
(\ref{sde}) reduces to (\ref{edssgs}) on each time interval $(T_n,T_{n+1})$.
One classically deduces that weak existence for 
(\ref{edssgs}) implies weak existence for (\ref{sde}),
see Ikeda-Watanabe \cite{iw} for similar considerations. 

\vip

We now prove the moment estimates. We have not found a direct proof
relying on stochastic calculus.
Fix $\beta \in (0,\alpha)$, $T>0$
and assume only that $b,\sigma$ have at most linear growth.
Consider a solution $(X_t)_{t\geq 0}$ to (\ref{sde}) and rewrite it as
in (\ref{sdeprime}).
Denote by $\cG=\sigma(T_1,T_2,\dots)$.
Then $X_t$ solves (\ref{edssgs}) during $[0,T_1)$. Hence we have
$$
\E\left[\left.\sup_{[0,T_1\land T)}X_t^2 \right\vert \cG\right]\leq C_T(1+x^2)
\quad \hbox{whence} \quad 
\E\left[\left.\sup_{[0,T_1\land T)}|X_t|^\beta  \right\vert \cG\right]
\leq K_T(1+|x|^\beta).
$$
Furthermore, $X_{T_1}=X_{T_1-}+\sigma(X_{T_1-})Z_1$, whence, since $\sigma$ 
has at most linear growth, 
$|X_{T_1}|^\beta\leq L(1+|X_{T_1-}|^\beta)(1+|Z_1|^\beta)$.
Consequently, we have
$$
\E\left[\left.\sup_{[0,T_1\land T]}|X_t|^\beta  
\right\vert \cG\right]\leq M_T(1+|x|^\beta)
$$
where $M_T=K_T+L(1+K_T)\E[1+|Z_1|^\beta]<\infty$ 
(here we need that $\beta<\alpha$ to have $\E[|Z_1|^\beta]<\infty$).
Exactly in the same way,
since $(X_t)_{t\geq 0}$ solves (\ref{edssgs}) 
during $(T_k,T_{k+1})$ for any $k\geq 1$, 
one can prove that
$$
\E\left[\left.\sup_{[T_{k}\land T,T_{k+1}\land T]}|X_t|^\beta \right\vert \cG\right]
\leq M_T(1+\E[|X_{T_{k}}|^\beta\vert \cG])
$$
with the same constant $M_T$.
Put $u_k=\E[\sup_{[T_{k}\land T,T_{k+1}\land T]}|X_t|^\beta\vert \cG]$
for $k\geq 0$ (set $T_0=0$).
We have proved that $u_0 \leq M_T(1+|x|^\beta)$ and that 
$u_{k+1} \leq M_T(1+u_k)$. We classically deduce that for some
constant $A_T>1$, depending on $x$, $u_k \leq A_T^{k+1}$. Consequently,
for any $k\geq 1$,
$$
\E\left[\left.\sup_{[0,T_k\land T]} |X_t|^\beta\right\vert \cG\right] 
\leq u_0+\dots+u_{k-1}
\leq A_T+\dots+A_T^k \leq \frac{ A_T^{k+1}}{A_T-1}.
$$
Finally, we find
\begin{align*}
\E\left[\sup_{[0,T]} |X_t|^\beta\right] \leq& \sum_{k \geq 0} 
\E\left[\indiq_{\{T_k<T<T_{k+1}\}} 
\E\left(\left.\sup_{[0,T_{k+1}\land T]} |X_t|^\beta\right\vert \cG \right)\right]\\
\leq & \frac{1}{A_T-1}\sum_{k \geq 0} A_T^{k+1} \frac{(\lambda T)^k}{k!} 
e^{-\lambda T} <\infty.
\end{align*}
This concludes the proof.
\end{preuve}

\end{document}